\documentclass[12pt]{amsart}

\usepackage{amsmath}
\usepackage{caption}
\usepackage[curve]{xypic}
\usepackage{cite}
\usepackage{enumitem}
\usepackage{color}
\usepackage{url}
\usepackage[usenames,dvipsnames]{xcolor}
\usepackage{graphicx}
\usepackage[all,cmtip]{xy}
\usepackage{psfrag}
\usepackage{hyperref} 
\usepackage{yhmath}
\usepackage{epstopdf}

\makeatletter 
\def\@cite#1#2{{\m@th\upshape\bfseries%
[{#1\if@tempswa{\m@th\upshape\mdseries, #2}\fi}]}}
\makeatother 

\theoremstyle{plain}
\newtheorem{thm}{Theorem}[section]
\newtheorem{cor}[thm]{Corollary}
\newtheorem{prop}[thm]{Proposition}
\newtheorem{lem}[thm]{Lemma}

\theoremstyle{definition}
\newtheorem{defn}[thm]{Definition}

\newtheorem{ex}[thm]{Example}

\theoremstyle{remark}
\newtheorem{rem}[thm]{Remark}

\numberwithin{equation}{subsection}
\renewcommand{\bold}[1]{\medskip \noindent {\bf #1 }\nopagebreak}

\newcommand{\nc}{\newcommand}
\newcommand{\rnc}{\renewcommand}
\newcommand{\e}{\varepsilon}

\nc\bA{\mathbb{A}}
\nc\bB{\mathbb{B}}
\nc\bC{\mathbb{C}}
\nc\bD{\mathbb{D}}
\nc\bE{\mathbb{E}}
\nc\bF{\mathbb{F}}
\nc\bG{\mathbb{G}}
\nc\bH{\mathbb{H}}
\nc\bI{\mathbb{I}}
\nc{\bJ}{\mathbb{J}} 
\nc\bK{\mathbb{K}}
\nc\bL{\mathbb{L}}
\nc\bM{\mathbb{M}}
\nc\bN{\mathbb{N}}
\nc\bO{\mathbb{O}}
\nc\bP{\mathbb{P}}
\nc\bQ{\mathbb{Q}}
\nc\bR{\mathbb{R}}
\nc\bS{\mathbb{S}}
\nc\bT{\mathbb{T}}
\nc\bU{\mathbb{U}}
\nc\bV{\mathbb{V}}
\nc\bW{\mathbb{W}}
\nc\bY{\mathbb{Y}}
\nc\bX{\mathbb{X}}
\nc\bZ{\mathbb{Z}}
\nc\cA{\mathcal{A}}
\nc\cB{\mathcal{B}}
\nc\cC{\mathcal{C}}
\rnc\cD{\mathcal{D}}
\nc\cE{\mathcal{E}}
\nc\cF{\mathcal{F}}
\nc\cG{\mathcal{G}}
\rnc\cH{\mathcal{H}}
\nc\cI{\mathcal{I}}
\nc{\cJ}{\mathcal{J}} 
\nc\cK{\mathcal{K}}
\rnc\cL{\mathcal{L}}
\nc\cM{\mathcal{M}}
\nc\cN{\mathcal{N}}
\nc\cO{\mathcal{O}}
\nc\cP{\mathcal{P}}
\nc\cQ{\mathcal{Q}}
\rnc\cR{\mathcal{R}}
\nc\cS{\mathcal{S}}
\nc\cT{\mathcal{T}}
\nc\cU{\mathcal{U}}
\nc\cV{\mathcal{V}}
\nc\cW{\mathcal{W}}
\nc\cY{\mathcal{Y}}
\nc\cX{\mathcal{X}}
\nc\cZ{\mathcal{Z}}
\nc\boldkappa{\boldsymbol\kappa}

\newcommand{\bk}{{\mathbf{k}}}

\nc{\dmo}{\DeclareMathOperator}
\rnc{\Re}{\operatorname{Re}}
\rnc{\Im}{\operatorname{Im}}
\nc{\hyp}{\operatorname{hyp}}
\nc{\even}{\operatorname{even}}
\nc{\odd}{\operatorname{odd}}
\rnc{\span}{\operatorname{span}}
\dmo{\rank}{rank}
\dmo{\End}{End}
\dmo{\Hom}{Hom}
\dmo{\Jac}{Jac}
\dmo{\Id}{Id}
\dmo{\Div0}{Div_0}
\dmo{\Aut}{Aut}
\dmo{\CP}{\bC P^1}
\dmo{\Ann}{\operatorname{Ann}}

\nc{\Sim}{{\mathop{}\!\sim}}

\nc{\BM}{{\overline{\mathcal M}}}
\nc{\BHH}{{\overline{\mathcal H}}}
\nc{\WHH}{{\widetilde{\mathcal H}}}
\nc{\HH}{{\mathcal H}}
\nc{\bbZ}{{\mathbb Z}}

\title{The WYSIWYG compactification}

\author{Dawei~Chen}
\address{Department of Mathematics, Boston College, Chestnut Hill, MA 02467, USA}
\email{dawei.chen@bc.edu}

\author{Alex~Wright}
\address{Department of Mathematics, University of Michigan, Ann Arbor, MI 48109, USA}
\email{alexmw@umich.edu}


\begin{document}
\maketitle
\thispagestyle{empty}

\begin{abstract}
We show that the partial compactification of a stratum of Abelian differentials previously considered by  Mirzakhani and Wright is not an algebraic variety. Despite this, we use a combination of algebro-geometric and other methods to provide a short, unconditional proof of  Mirzakhani and Wright's formula for the tangent space to the boundary of a $GL^{+}(2,\bR)$ orbit closure, and give new results on the structure of the boundary.
\end{abstract}

\setcounter{tocdepth}{1} 
\tableofcontents

\section{Introduction}\label{S:intro}

\bold{WYSIWYG.} Let $\kappa$ be a multiset of $n$ non-negative integers that sum to $2g-2$. Define the stratum $\cH(\kappa)$ to be the space of genus $g$ Riemann surfaces $X$ with $n$-unordered marked points and an Abelian differential $\omega\in H^{1,0}(X)$ that vanishes at the marked points with orders of vanishing given exactly by $\kappa$. The multiset $\kappa$ may contain $0$ with some multiplicity, which corresponds to marked points where the differential does not vanish, or ``zeros of order zero". (We mark all the zeros of $\omega$ so that the true zeros can be treated on an equal footing with the ``zeros of order zero".) The stratum is contained in the Hodge bundle over the moduli space $\cM_{g,n}$ of genus $g$ Riemann surfaces with $n$ unordered marked points. 

Let $\BHH(\kappa)$ be the closure of $\cH(\kappa)$ in the Hodge bundle over the Deligne-Mumford moduli space $\BM_{g,n}$ of stable genus $g$ Riemann surfaces. 

For $(X,\omega)\in \BHH(\kappa)$, define $\pi(X,\omega)$ to be the result of deleting each component of $X$ on which $\omega$ is identically zero, deleting each node, and filling in any punctures thus created with marked points. Thus $\pi(X, \omega)$ is a union of Riemann surfaces with marked points, each of which has a non-zero Abelian differential which is holomorphic away from the marked points and has at worst simple poles at the marked points. Notice that $(X,\omega)$ and $\pi(X, \omega)$ have finite area under the flat metric induced by $\omega$ if and only if $\omega$ has no simple pole at any node of $X$. 

Define an equivalence relation on $\BHH(\kappa)$ by $(X,\omega) \sim (X', \omega')$ if $\pi(X,\omega)=\pi(X', \omega')$, and define the What You See Is What You Get partial compactification $\WHH(\kappa)$ to be $\BHH(\kappa)/{\sim}$. This comes equipped with the quotient  map $\pi\colon \BHH(\kappa) \to \WHH(\kappa)$.

We also define the projectivizations $\bP \BHH(\kappa)$ and $\bP \WHH(\kappa)$ of these spaces to be the quotient by the $\bC^*$ action that rescales the differential. 

The WYSIWYG compactification is so named because, when an Abelian differential is presented via polygons in the plane, and when these polygons degenerate in a reasonable way, the limit will be given by the resulting polygons; see \cite[Definition 2.2, Proposition 9.8]{MirWri}. 

\bold{A non-algebraic, non-analytic space.} Our first result cements the impression that the WYSIWYG compactification is a creation of flat geometry rather than algebraic geometry; it is in fact not an algebraic variety at all. For simplicity we restrict to the case where there are no zeros of order zero. 

\begin{thm}\label{T:notalg}
For $g\geq 3$ and every $\kappa$ consisting of positive integers, as well as $g=2$ and $\kappa = (1,1)$, $\bP\WHH(\kappa)$ does not admit the structure of an algebraic variety in such a way that $\pi\colon \bP\BHH(\kappa)\to \bP\WHH(\kappa)$ is a morphism of algebraic varieties. Moreover, $\bP\WHH(\kappa)$ is not even a complex analytic space in such a way that $\pi$ is a morphism of complex analytic spaces. When $\bP\cH(\kappa)$  is disconnected, the same is true for each connected component. 
\end{thm}

This result was inspired by the surprisingly recent result that the set of birational automorphisms of $\bP^n$ of degree $\leq d$ is not a variety in such a way that the quotient map from the variety of formulas of degree $\leq d$ birational maps is a morphism, even though the space of birational automorphisms of degree exactly equal to $d$ is a variety \cite{cremona}. In both cases we see that the space is the disjoint union of varieties, but in order to have any connection to moduli problems these varieties must be combined using a topology in such a way that the result cannot be a variety. 

 Another analogous situation occurs for the topological compactification $K_1\cM_{g,1}$ of the moduli space of one-pointed genus $g$ Riemann surfaces $\cM_{g,1}$ used by Kontsevich to prove Witten's conjecture \cite{KontWitt}.  Keel showed that the quotient map $\BM_{g,1}\to K_1\cM_{g,1}$ cannot be an algebraic morphism \cite{Keel}.  

Despite Theorem \ref{T:notalg},  questions about the WYSIWYG compactification are  amenable to  algebraic methods; Theorem \ref{T:notalg} simply indicates that some care may be required.  

\bold{Orbit closures.}  The finite area locus $\WHH_{<\infty}(\kappa)$ of  $\WHH(\kappa)$ was studied in \cite{MirWri} in order to  facilitate the inductive study of $GL^{+}(2,\bR)$ orbit closures. This partial compactification allows one to pass to the boundary of an orbit closure while staying in the realm of finite area Abelian differentials (possibly with marked points and possibly with multiple components).

{We define a stratum of multi-component translation surfaces as follows. Given a multiset $\boldkappa = \{\kappa_1, \ldots, \kappa_k\}$ of multisets $\kappa_i$ as above, define the stratum $\cH(\boldkappa)$ to be the product $\cH(\kappa_1)\times \cdots \times \cH(\kappa_k)$ modulo the subgroup of the symmetric group $S_k$ that fixes $(\kappa_1, \ldots, \kappa_k)$. An element of $\cH(\boldkappa)$ will be called a multi-component surface; the effect of quotienting by the subgroup of $S_k$ is that the $k$ components are not ordered. 
These strata appear  in the finite area locus in the boundary of $\cH(\kappa)$ in $\WHH(\kappa)$; a degeneration of connected surfaces may have several components.}

Recall that $GL^{+}(2, \bR)$ orbit closures may intersect themselves; at such atypical points there are finitely many branches of the orbit closure. Later we will discuss invariant subvarieties of multi-component surfaces; the following result concerns orbit closures of connected (single component) surfaces. 

\begin{thm}\label{T:main}
Let $\cM$ be a $GL^{+}(2, \bR)$ orbit closure in $\HH(\kappa)$ and let $\partial{\cM}_{<\infty}$ be its boundary in $\WHH_{<\infty}(\kappa)$. Suppose that $(X_n, \omega_n)\in \cM$ converge to $(X_\infty, \omega_\infty) \in \partial{\cM}_{<\infty}$.

The intersection $\cM'$  of $\partial{\cM}_{<\infty}$ and the stratum of $(X_\infty, \omega_\infty)$ is an algebraic variety locally described by a finite union of linear subspaces in local period coordinates. 

After removing finitely many terms, the sequence  $(X_n, \omega_n)$ may be partitioned into finitely many subsequences such that for each subsequence the tangent space to a branch of $\partial{\cM}_{<\infty}$ in the boundary stratum   at $(X_\infty, \omega_\infty)\in \WHH_{<\infty} $ is equal to the intersection of the tangent space of a branch of $\cM$ at $(X_n, \omega_n)$ and the tangent space of the boundary stratum, for $n$ sufficiently large.  
\end{thm} 

The tangent space to the boundary stratum naturally sits inside the tangent space to $\cH$ at $(X_n, \omega_n)$ for $n$ sufficiently large; this is explained in Section 2 and \cite{MirWri}. Note that $X_n$ and $X_\infty$ denote Riemann surfaces with a finite set of marked points, and that $X_\infty$ may have zeros of order zero  (marked points where the differential does not vanish) even if the $X_n$ do not.

In the case when $X_\infty$ has only one component, Theorem \ref{T:main} was proven in \cite{MirWri}. In the case when $X_\infty$ has multiple components, a proof was given in \cite{MirWri} that was conditional on an anticipated but still unproven generalization of \cite{EM, EMM} to multi-component surfaces. Here we give an unconditional proof using different techniques, thus giving the first complete proof in the multi-component case, which is crucial for applications.  

The new proof is much simpler and completely different, although it will use a non-trivial foundational result on $\WHH$ from \cite{MirWri}. The main difference is that here we use  Filip's result that orbit closures are algebraic varieties \cite{Fi2}. In light of Filip's work, we may use ``invariant subvariety" as a synonym for $GL^{+}(2,\bR)$ orbit closure. When we do so, it is implicit that the invariant subvariety is irreducible.

\bold{Other points of view.} Theorem~\ref{T:main} has an interpretation in terms of flat geometry. An orbit closure $\cM$ is locally cut out by linear equations on period coordinates of the ambient stratum, and period coordinates of an Abelian differential corresponds to edges of a polygonal presentation. If some edges of the polygon shrink to length zero, the equations on the limit surface can be obtained by replacing the variables for any edges that are collapsed with zero. 

A more abstract point of view is also possible in terms of the linear equations locally defining $\cM$. Each equation can be interpreted as a relative homology class. Theorem~\ref{T:main} gives that the linear equations defining the boundary orbit closure arise from pushing forward those defining $\cM$ via a collapse map, which will be introduced in Section \ref{S:compactification}.

\bold{The structure of the multi-component boundary.} 
Strata of the boundary of $\cM$ where the translation surfaces have multiple components have some especially interesting and useful properties. We will show that each stratum of the boundary of $\cM$ is again an invariant subvariety, in that it is an algebraic variety that is invariant under a natural $GL^{+}(2,\bR)$ action. 

Each stratum of multi-component surfaces is a quotient of a product of strata by a subgroup of a symmetric group. Given an invariant subvariety of such a stratum, one can ``lift" it to a product of strata by taking an irreducible component of its preimage, and it is often convenient to do so. This lift is well-defined up to the ordering of the components.

An invariant subvariety $\cM\subset \cH(\kappa_1)\times \cdots \times \cH(\kappa_k)$ of a product of strata is called prime unless, possibly after reordering the components, there is some $1\leq s< k$, and invariant subvarieties $\cM' \subset \cH(\kappa_1)\times \cdots \times\cH(\kappa_s)$ and $\cM'' \subset \cH(\kappa_{s+1})\times \cdots\times \cH(\kappa_k)$ such that $\cM=\cM'\times \cM''$. More generally, an invariant subvariety of a stratum of multi-component surfaces is prime if its lift is prime. So, any invariant subvariety can (after lifting if required) be written as a product of primes. The definition is clarified in Section \ref{S:multi}, where we also give a simple linear algebraic characterization. 

We provide the following result so that it can be used with Theorem \ref{T:main} in forthcoming work towards the goal of classifying orbit closures.  

\begin{thm}\label{T:prime}
Let $\cN$ be a prime invariant subvariety of a product of strata. Then 
\begin{enumerate}
\item if $\pi_i$ is  the projection to the $i$-th component, there is an irreducible invariant subvariety $\cN_i$ and a finite union $\cN_i'$ of irreducible invariant subvarieties properly contained in $\cN_i$ such that $$\cN_i-\cN_i'\subset \pi_i(\cN) \subset \cN_i,$$
\item locally in $\cN$, the absolute periods of any component of a multi-component surface 
 determine the absolute periods of any other component, 
\item each $\cN_i$ has the same rank, and 
\item certain natural factors of the Jacobians of the different components of any surface in $\cN$ are isogenous. 
\end{enumerate}
\end{thm}
Part of this theorem was anticipated in \cite[Conjecture 2.10]{MirWri}. We anticipate that the second point will be the most useful for classification. 
 
\bold{Additional context.} Compactifications of strata are currently an area of great interest, see for example \cite{Gen, Chen, Many, bainbridge2016strata, chen2016principal, chen2016spin, farkas2016moduli}. 
It would be interesting to understand the boundary of a $GL^{+}(2, \bR)$ orbit closure in the larger compactifications studied by these works.

The study of orbit closures via their boundary is already proving successful, see for example \cite{mirzakhani2016full, Apisa2, aulicino2016rank}.

A version of $\WHH$ appeared previously in the work of McMullen \cite{McM:nav}. 

{One could also define a version of $\WHH$ that remembered finitely much more information, namely which of the marked points were glued together after the collapse, but we chose not to do so.} 

\bold{Acknowledgments.} The authors thank Paul Apisa, Francisco Arana Herrera, Matt Bainbridge, Brian Conrad, Ben Dozier, Simion Filip, Francois Greer, Brendan Hassett, Tasho Kaletha, Maryam Mirzakhani, and Martin M\"{o}ller for useful conversations. The authors especially thank the referee for many helpful comments.

During the preparation of this work the first named author was partially supported by an NSF CAREER Award DMS-1350396, an NSF Standard Grant DMS-2001040, a Simons Collaboration Grant 635235 and a von Neumann Fellowship at IAS, and the second named author was partially  supported by a Clay Research Fellowship.

\section{The partial compactification}\label{S:compactification}

\begin{lem}\label{L:iff}
A sequence $(X_n, \omega_n)\in \cH$ converges to $(X, \omega)$ in $\WHH$ if and only if, for every subsequence $n_k$ such that $X_{n_k}$ converges in the Deligne-Mumford compactification, the sequence $(X_{n_k}, \omega_{n_k})$ breaks into a union of finitely many subsequences each of which converges in $\BHH$ to a limit $(X', \omega')\in \BHH$ that satisfies  $\pi(X', \omega')=(X,\omega)$. (The limit $(X',\omega')$ may depend on the subsequence.)
\end{lem}

An irreducible component $Y$ of $(X', \omega')\in \BHH$ has two possible types, depending on whether $\omega'$ is identically zero or not on $Y$. We call the former type a ``collapsed component''. The idea of the proof of Lemma~\ref{L:iff}  is that the data of $(X', \omega')$ is the data of $\pi(X', \omega')$ plus the additional data of the collapsed components and how the components are glued together at nodes. This additional data is recorded in  the Deligne-Mumford compactification. In other words, the images of $(X', \omega')$ in $\WHH$ and the Deligne-Mumford compactification almost determine $(X', \omega')\in\BHH$ uniquely. The only situation in which $(X', \omega')\in\BHH$ is not determined uniquely is when there are multiple components of the limit in the Deligne-Mumford compactification that are isomorphic. In this case, there may be finitely many  $(X', \omega')\in\BHH$, corresponding to different permutations of non-zero differentials on isomorphic components of $X$. 

\begin{proof}[Proof of Lemma~\ref{L:iff}]
For $(X, \omega) \in \WHH$, we first analyze its preimage in $\BHH$. Suppose that $(X, \omega)$ consists of $(Y_1, \omega_1), \ldots, (Y_m, \omega_m)$,  where $Y_1, \ldots, Y_m$ are irreducible components of $X$ and $\omega_i$ is the restriction of $\omega$ to $Y_i$, which is non-zero. Then $\pi^{-1}(X, \omega)\subset \BHH$ is a compact subset given by the union of certain products of Deligne-Mumford compactifications parameterizing the collapsed components. In particular, all $(X', \omega') \in \pi^{-1}(X, \omega)$ contain the same set of non-collapsed components $(Y_i, \omega_i)$ and only the collapsed components can vary. 

Suppose that a sequence $(X_n, \omega_n)$ converges to $(X, \omega)$ in $\WHH$. Let $n_k$ be a subsequence such that $X_{n_k}$ converges to $X'$ in the Deligne-Mumford compactification. Then $X'$ and $\omega$ determine finitely many possible limits $(X', \omega^{(i)}) \in \pi^{-1}(X, \omega)$, where $\omega^{(i)} = \omega$ on the non-collapsed components, $\omega^{(i)} = 0$ on the collapsed components, and the $\omega^{(i)}$ differ from each other by permuting differentials on isomorphic non-collapsed components. Thus the sequence $(X_{n_k}, \omega_{n_k})$ breaks into finitely many subsequences, each of which converges to one such limit  $(X', \omega^{(i)})$.   

Conversely, suppose for every subsequence $n_k$ such that $X_{n_k}$ converges in the Deligne-Mumford compactification, the sequence $(X_{n_k}, \omega_{n_k})$ 
breaks into a finite union of subsequences  each of which converges in $\BHH$ and the limit $(X', \omega')\in \BHH$ satisfies $\pi(X', \omega{'})=(X,\omega)$.  
{Since $\pi$ is continuous, this implies that the subsequence of $(X_{n_k}, \omega_{n_k})$ converges to $(X,\omega)$. Compactness of the Deligne-Mumford compactification now gives the result, since every subsequence of $X_n$ has a sub-subsequence that converges.}
%
\end{proof}

\begin{rem}
{Lemma \ref{L:iff} shows in particular that every sequence in $\cH$ has a unique limit in $\WHH$. The same analysis applies to nets, with the same conclusion, and one can extend the analysis to nets in $\WHH$. (Without loss of generality, the set of collapsed components can be assumed to be constant along the net.) Since a space is Hausdorff if and only if convergent nets have unique limits, this shows $\WHH$ is Hausdorff. 

An alternate approach to showing that $\WHH$ is Hausdorff is to show that $\pi$ is closed and invoke \cite[exercise 7 on page 199]{Munkres}. To see that $\pi$ is closed, for each closed set $E \in \BHH$, it is not hard to show that $\pi^{-1}(\pi(E))$ is closed. (Here it suffices to consider sequences, and without loss of generality the set of collapsed components can be assumed to be constant along the sequence.)}

Once $\bP\WHH$ is known to be Hausdorff, it follows  that it is metrizable \cite[Corollary 23.2]{Willard}. Hence $\WHH$ is also metrizable, since it is a $\mathbb C^{*}$ bundle over $\bP\WHH$. In particular, using nets is not required.
\end{rem}

%
%
%
%

The second criterion for convergence below forms part of our basic understanding of $\WHH$. It says that a sequence converges if and only if the flat metrics converge away from the part of the surfaces that is shrinking. 

\begin{thm}[Mirzakhani-Wright]
$(X_n, \omega_n)\in \cH$ converge to $(X, \omega)$ in $\WHH$ if and only if there are decreasing neighborhoods $U_n\subset X$ of the set of marked points $S$ with $\cap U_n=S$ and maps $g_n\colon X\setminus U_n\to X_n$ that are diffeomorphisms onto their range, such that 
\begin{enumerate}
\item $g_n^*(\omega_n) \to \omega$ in the compact open topology on $X\setminus S$, and 
\item the injectivity radius of points not in the image of $g_n$ goes to zero uniformly in $n$.
\end{enumerate}
\end{thm}
The injectivity radius at a point is defined as the sup of all $\e>0$ such that the ball of radius $\e$ in the flat metric is embedded and does not contain any marked points or singularities. For the proof, see \cite[Definition 2.2, Proposition 9.8]{MirWri}. 

\begin{cor}
 The action of $GL^{+}(2,\bR)$ extends continuously to $\WHH$. 
\end{cor}

{If a sequence $(X_n, \omega_n)$ in $\HH$ converges to $(X, \omega)$ in $\WHH$}, then there are natural collapse maps $f_n\colon X_n\to X'$, defined for $n$ sufficiently large. Here $X'$ is $X$ with some additional identification between marked points; passing to a subsequence of $(X_n, \omega_n)$ that converges in $\BHH$, we can describe $X'$ as the limit in $\BHH$ with each zero area component collapsed to a point.  (The collapse maps are only well defined up to pre- and post-compositions of automorphisms. This is related to the fact that all the moduli spaces under consideration are only orbifolds.) 

Note that if $\Sigma'$ is a finite subset of $X'$ containing the nodes, and $\Sigma$ is its pre-image on $X$, then $H_1(X,\Sigma)\simeq H_1(X', \Sigma')$. This allows us to conflate $X$ and $X'$ when working on the level of relative (co)homology. 

The maps $f_n$ can be chosen to map zeros of $\omega_n$ to zeros or marked points. If $\Sigma_n$ is the set of zeros of $\omega_n$, we obtain well defined maps 
$$f_n^*\colon H^1(X, \Sigma) \to H^1(X_n, \Sigma_n) \quad\text{and}\quad (f_n)_*\colon H_1(X_n, \Sigma_n) \to H_1(X, \Sigma).$$ 
The first map is injective and the second surjective. Hence $f_n^*$ naturally identifies the tangent space $(X, \omega)$ in its stratum with the image of $f_n^*$, a subspace of the tangent space of the stratum of $(X_n,\omega_n)$. 

This subspace can also be identified with the annihilator $\Ann(V_n)$ of the subspace $V_n=\ker((f_n)_*)$ of relative homology. We call $V_n$ the space of vanishing cycles, although some care should be taken since this term might equally well be used to refer to a similar object in absolute homology. 

Since so much collapsing can occur, we do not  know of a nice way to describe the local structure of $\WHH$ near a boundary point. However we have the following lemma indicating that the situation is not too pathological from at least one point of view. 

\begin{prop}\label{P:xi}
Let $(X,\omega)\in\WHH_{<\infty}$. Then
$\Ann(V)$ can be defined on a neighborhood of $(X,\omega)$ in $\cH$ and is locally constant. There is a neighborhood of $0$ in $\Ann(V)$ such that if $\xi_n, \xi$ are in this neighborhood and $\xi_n\to \xi$, then $(X_n, \omega_n)+\xi_n$ is well-defined and converges to $(X,\omega)+\xi$. 
\end{prop}

Here we use the notation ``$+\xi$" to denote the surface whose period coordinates differ from the original surface by $\xi$; this is only defined for $\xi$ small. This proposition, which appears as \cite[Proposition 2.6]{MirWri}, rules out the pathological possibility that the stratum could get skinnier and skinner as it approaches the boundary point $(X,\omega)$, so that only increasingly tiny deformations in the $\Ann(V)$ directions are well defined. 

{We end the section with a remark.}

\begin{rem}\label{R:dual}
As we hinted at in the introduction, one could state Theorem \ref{T:main} in a dual form, which would say that the space of equations defining a boundary  invariant subvariety is the image under the collapse map $(f_n)_*$ of the set of equations defining the invariant subvariety. 
\end{rem}

\section{Complex analytic spaces}\label{S:rigid}

We first recall the definition of a complex analytic space \cite{GR}.

\begin{defn} 
A $\bC$-space is a pair $(X,\cO_X)$, where $X$ is a topological space and $\cO_X$ is a sheaf of $\bC$-algebras on $X$. A morphism $(X, \cO_X)\to (Y,\cO_Y)$ of $\bC$-spaces is a continuous map $f\colon X\to Y$ together with a collection of $\bC$-algebra maps $\cO_Y(V)\to \cO_X(f^{-1}(V))$ for each open subset $V\subset Y$ that commute with restriction maps. The morphism is called an isomorphism if $f$ is a homeomorphism and all the maps  $\cO_Y(V)\to \cO_X(f^{-1}(V))$ are isomorphisms. 
\end{defn}

\begin{defn} 
A $\bC$-space $(X,\cO_X)$ is a complex analytic space if it is locally isomorphic to the vanishing locus of a finite collection $f_1, \ldots, f_k$ of holomorphic functions defined on an open subset $U$ of some affine space $\bC^n$, equipped with the sheaf of holomorphic functions on $U$ modulo the ideal $(f_1, \ldots, f_k)$. A morphism of complex analytic spaces is a morphism of their underlying $\bC$-spaces where the maps on sheaves are given by pull back of functions. 
\end{defn}

\begin{ex} 
Every quasi-projective variety over $\bC$ is in particular a complex analytic space. A morphism of quasi-projective varieties is in particular a morphism of complex analytic spaces. 
\end{ex}

\begin{lem}\label{L:rigidity}
{Let $M$ and $N$ be connected complex manifolds such that $N$ is compact}, and let $X$ be a complex analytic space.  
If a morphism $f\colon M\times N \to X$ maps $\{m\}\times N$ to a point for some $m\in M$, then $f$ factors through the projection $M\times N \to M$. 
\end{lem}

A version of this result, sometimes called the Rigidity Lemma, holds for algebraic varieties. Presumably Lemma \ref{L:rigidity} is also known, but we do not know a reference. 

\begin{proof}
Let $S$ be the set of $m\in M$ such that $\{m\}\times N$ is mapped to a point. Note that $S$ is closed in $M$. 

{Since $M\times N$ is connected, the image of $f$ is contained in a connected component of $X$. Hence without loss of generality we can assume that $X$ is connected. }
We also assume that $X$ is not a point, for otherwise the claim is trivial. 

Suppose in order to find a contradiction that we can pick $m_0\in \partial S$. 
Since $X$ is a  complex analytic space, every point in $X$ has a neighborhood $U$ such that for any pair of distinct points $p,q\in U$ there exists $g\in \cO_X(U)$ (i.e. a holomorphic function $g\colon U\to \bC$) such that $g(p)\neq g(q)$. Let $U$ be such a neighborhood of the point $f(\{m_0\}\times N)$. The open set $f^{-1}(U)$ contains $\{m_0\}\times N$. Since $N$ is compact, we can find an open set $U_M$ of $M$ containing $m_0$ such that $U_M \times N \subset f^{-1}(U)$. Pick $m_1\in U_M\setminus S$, and pick $n_1, n_2\in N$ such that $f(m_1, n_1)\neq f(m_1, n_2)$. Choose $g\in \cO_X(U)$ such that $g\circ f(m_1, n_1)\neq g\circ f(m_1, n_2)$. The function $g\circ f$ restricted to $\{m_1\}\times N$ is a non-constant holomorphic function on a connected compact complex manifold, which is a contradiction. 
\end{proof}

\begin{proof}[Proof of Theorem \ref{T:notalg}] 
We first prove the claim for each stratum, regardless of whether it is connected or not, and then later we will address the individual connected components. 

Let $\cH(\kappa_1, \ldots, \kappa_s, -2, -2)$ be the stratum of meromorphic differentials that have zeros $z_i$ of type $\kappa$ and two double poles $p_1$ and $p_2$. As long as $\kappa \neq (2)$, by \cite{GenTah} there exists a differential $(C, \xi)$ in this stratum such that $\xi$ has zero residue at $p_1$ and $p_2$.  Attach $C$ to a fixed genus one curve $E_1$ at $p_1$, where $E_1$ carries a fixed non-trivial differential $\omega_1$.  Next take a pencil $B\cong \bP^1$ of plane cubics $E_{b}$,  
where $b\in B$ parametrizes the elliptic curves in the family, and attach to $p_2$ the curve $E_b$. Then we obtain a family of genus $g$ nodal curves parameterized by $B$; see Figure~\ref{fig:family}.
\begin{figure}[ht!]
\includegraphics[width=.4\linewidth]{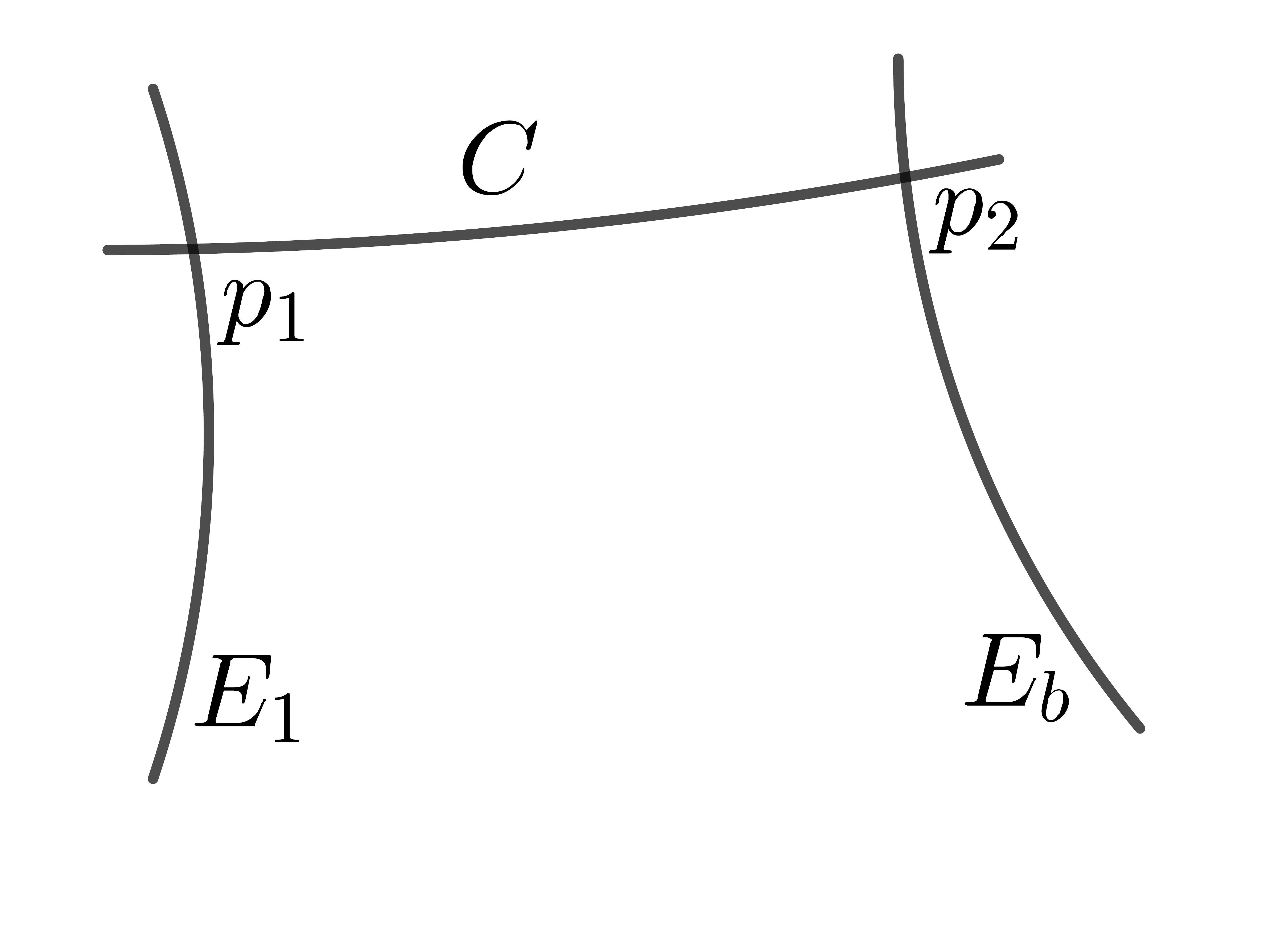}
\caption{\label{fig:family} The underlying nodal curves in the construction}
\end{figure}
Note that the Hodge bundle $\BHH$ restricted to $B$ is isomorphic to $\cO_{\bP^1}(1) \oplus \cO_{\bP^1}^{\oplus (g-1)}$, 
as $C$ and $E_1$ do not vary and in the case of genus one the Hodge bundle on a pencil of plane cubics has degree one; see e.g. \cite[Chapter 3 F]{HarrisMorrison}.  
Let $\sigma$ be a non-trivial global section of the $\cO(1)$ part. Then $\sigma$ has only one zero at a point $b_0\in B$. Let $\omega_{b} = \sigma (b)$ be the corresponding differential on $E_b$. It follows that $\omega_b$ is non-trivial on all $E_b$ but $E_{b_0}$ (as $\sigma(b_0)$ is zero). 

Let $[u, v]$ be the homogeneous coordinates of $\bP^1$. For each fixed value of $[u,v] \neq [0,1]$, define a one-parameter family of stable differentials   
given by 
$$(E_1, u\omega_1), \ (C, 0), \ (E_{b},  v \omega_{b})$$ 
where $b$ varies in $B$. Since the scalings of $\omega_1$ and $\omega_b$ depend on $[u,v]$, we denote by 
$B_{[u, v]}$ the base for the parameterization of the above family.   
By the zero residue assumption 
at $p_1$ and $p_2$, we know that stable differentials parameterized by $B_{[u, v]}$ are contained in $\bP\BHH(\kappa)$ according to \cite{Many} (as the global residue condition therein holds). We require $[u,v] \neq [0,1]$ because the corresponding differential at $b_0$ is identically zero on every component of the underlying curve (hence not well-defined in $\bP\BHH$). Now varying $[u,v]$ as well in $\bC\cong \bP^1 \setminus [0,1]$, the union of these $B_{[u, v]}$ forms a complex two-dimensional family isomorphic to $\bC \times \bP^1$, where the base $\bC$ has coordinates $[u,v]$ and the fiber $\bP^1$ over $[u,v]$ corresponds to $B_{[u,v]}$. Note that $B_{[u,v]}$ is contracted by $\pi$ if and only if $[u, v] = [1,0]$, as it gets contracted if and only if the differentials on $E_b$ are identically zero. We thus conclude the proof by using Lemma \ref{L:rigidity}, where in that context $M = \bC$, $N = \bP^1$, $X = \bP\WHH(\kappa)$, and  $m = [1,0]\in M$.

If $\cH(\kappa)$ is disconnected, then the above proof also applies to each connected component of $\cH(\kappa)$. For example, consider 
$\cH(2g-2)$, which has three connected components 
$\cH(2g-2)^{\hyp}$, $\cH(2g-2)^{\even}$ and $\cH(2g-2)^{\odd}$ when $g > 3$. If the genus $g-2$ curve 
$C$ is hyperelliptic and if the unique zero $z$ and the two poles $p_1$, $p_2$ on $C$ are all Weierstrass points (or $2$-torsion to each other if $C$ is of genus one), then the constructed family $B_{[u, v]}$ lies in $\bP\BHH(2g-2)^{\hyp}$. For the spin components, since the nodal curves in the family are of compact type, their spin parity will be determined by the parity of $C$ only, as the parity of the two elliptic tails does not vary. Hence one can take the underlying curve $(C, z, p_1, p_2)$ from a meromorphic differential in the respective spin components of $\cH(2g-2, -2, -2)$ according to \cite{KZ:comps, Boissy}.  The other cases are similar and we omit the details.  
\end{proof}

\begin{rem}
When $\bP\cH(\mu)$ is disconnected, the closures of its connected components in the Hodge bundle or in the Deligne-Mumford moduli space may intersect each other in the boundary. See \cite{Chen, Gen} for some examples. Nevertheless, one can add log structures to help distinguish the components in the boundary \cite{chen2016spin}. 
\end{rem}

\begin{rem}
The above proof does not apply to the remaining case $\kappa = (2)$ in genus two, because there is no differential on $\bP^1$ with a unique zero and at least two poles such that all residues are zero (see \cite[Lemma 3.6]{Many} and \cite{GenTah}). 
\end{rem}

\begin{rem}
Following the previous remark, indeed $\bP\WHH(2)$ can be given a natural variety structure such that it is compatible with the other algebraic compactifications.  
To see this, we slightly abuse notation to denote by $\bP\BHH(2)$ the compactification in \cite{Many}. Namely, we mark the unique zero $z$ and take the closure 
of $\bP \HH(2)$ in the projectivized Hodge bundle over $\BM_{2,1}$. Note that $\bP\BHH(2)$ can be identified via hyperelliptic (admissible) double covers with the moduli space $\BM_{0,1;5}$ of stable rational curves with six marked branch points where one of the markings is distinguished and the others are unordered, and the distinguished marking $z_1$ corresponds to the image of the unique double zero $z$ in the domain of the covers.

If $(X, \omega, z)$ in $\bP\BHH(2)$ has a genus one component $C$ containing $z$, then $\omega$ is identically zero on $C$.  This is because $\omega$ restricted to $C$ must have a double zero (at $z$), and can have at most a single simple pole (at the node), and no such non-zero $\omega$ exists. 
Under the hyperelliptic (admissible) cover, the image of $C$ in the target is a rational tail that contains the distinguished marking $z_1$ and two other unordered markings. Hence forgetting the moduli of $C$ is the same as forgetting the moduli of the image  
 rational tail that contains $z_1$ and two other markings (where the moduli corresponds to the cross ratio of the three markings together with the attaching node). By examining the boundary strata under the map $\pi\colon \bP\BHH(2) \to \bP\WHH(2)$, one can see that the above phenomenon is the only difference between the two spaces.

Let $\BM_{0,\cA}$ be the Hassett weighted moduli space $\BM_{0,\cA}$ \cite{Hassett} where $\cA$ assigns weight $\epsilon$ to $z_1$ for $0< \epsilon \ll 1$ and assigns weight $\frac{1}{2} - \epsilon$ to the other markings. The contraction morphism $f\colon\BM_{0,1;5} \to \BM_{0,\cA}$ exactly forgets the moduli of a rational tail that contains $z_1$ and two other markings, since the sum of their weights is less than one.  Hence we can identify $\bP\WHH(2)$ with $\BM_{0,\cA}$ and identify the map 
$\pi$ with $f$, where $\BM_{0,\cA}$ is an algebraic variety and $f$ is an algebraic morphism.  
\end{rem}

\section{Cautionary examples}\label{S:caution}

Our main cautionary example concerns a discontinuity in the behavior of periods for a certain straightforward degeneration of translation surfaces. Theorem \ref{T:main} shows in particular that limits of surfaces in $\cM$ satisfy limiting equations on their periods, and our main cautionary example shows that this is far from automatic.

Consider surfaces as in Figure~\ref{F:Discontinuous} that result from gluing in two horizontal cylinders into slits on a surface in $\cH(2)$. 
\begin{figure}[ht!]
\includegraphics[width=\linewidth]{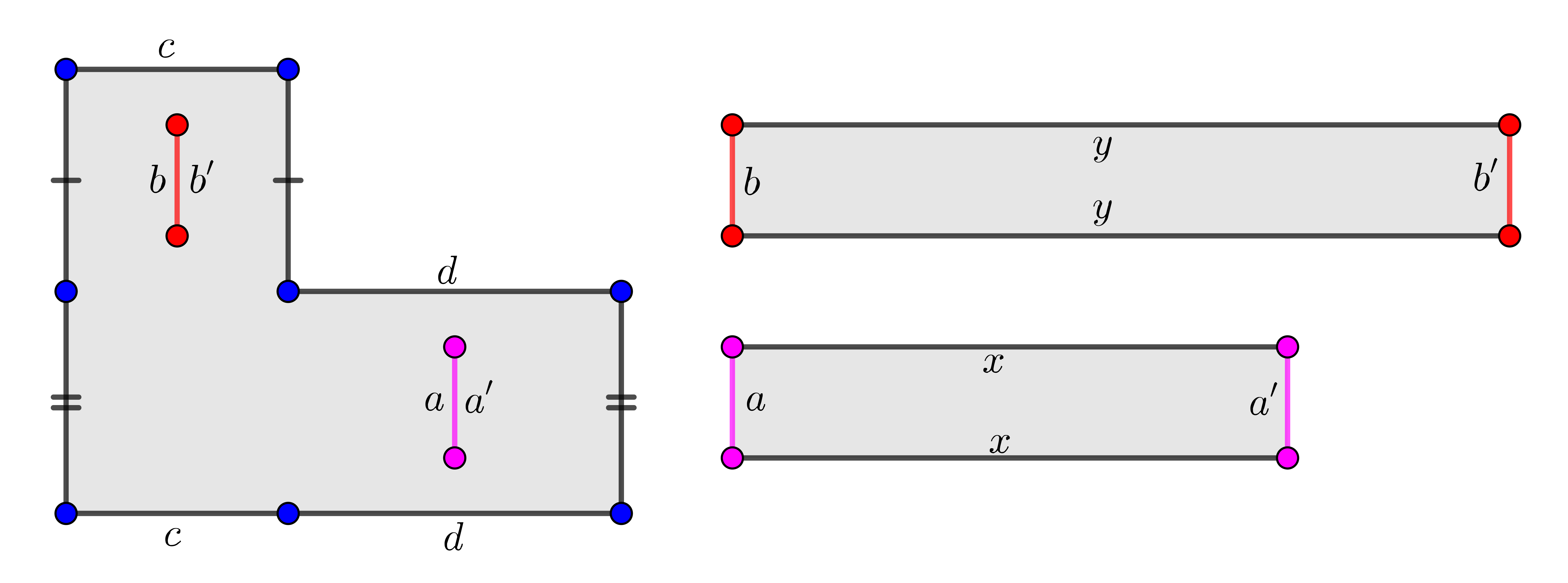}
\caption{A surface with cylinders of large modulus}
\label{F:Discontinuous}
\end{figure}
We will consider a family $(X_\e, \omega_\e)$ of such surfaces, parameterized by $\e \in (0, \e_0)$, such that 
$$a=a'=b=b'=i\epsilon$$ 
and 
$$x=\frac1\e, \quad\quad y-x=d-c.$$
Here we are using the label of the edge to also refer to the period coordinate of that edge.  All other edges periods are constant along the family, and we assume $c\neq d$. 

 Let $(X_0, \omega_0)$ denote the limit in $\WHH$, and define the periods $x,y$ etc on the limit by pushing forward the relative homology classes and then taking the period on the limit $(X_0, \omega_0)$. 

All of these surfaces  $(X_\e, \omega_\e)$ satisfy the equation on period coordinates $y-x=d-c$. On the limit surface, the periods $x$ and $y$ are now zero, and since $c$ and $d$ are constant and $c\neq d$ we have 
$y-x\neq d-c$. That is, the limit surface does not satisfy the limit equation, even though the equation holds identically on the degenerating family $(X_\e, \omega_\e)$! 

In our situation, we worry that there might be some invariant subvariety $\cM$ containing all $(X_\e, \omega_\e)$ and defined by the equation $y-x=d-c$ in local period coordinates. The existence of such an $\cM$ would contradict Theorem \ref{T:main}. 
This contradiction is most immediate from the point of view of Remark \ref{R:dual}, because the limit surface does not satisfy the limit equation. (One can also see the contradiction by noting that the tangent direction to the boundary orbit closure $\cM'\ni (X_0, \omega_0)$ resulting from scaling $\omega_0$ would not lie in $T_{(X_\e, \omega_\e)}(\cM)$, because that scaling changes $c-d$ without changing $x$ or $y$.)  

However, it is easy see that no such $\cM$ exists, as follows.

\begin{lem}
There does not exist an invariant subvariety $\cM$ containing all $(X_\e, \omega_\e)$ such that the equation $y-x=d-c$  holds locally on $\cM$.
\end{lem}

\begin{proof}
Otherwise, the Cylinder Deformation Theorem of \cite{Wcyl}, stated as Theorem \ref{T:CDT} below, would give a deformation that scaled $x$ and $y$ while fixing $c$ and $d$. This contradicts the equation $y-x=d-c$  because we have assumed that $c-d\neq 0$.

We now give more details on this argument, using the notation and terminology defined shortly in Section \ref{S:parallel}. Because all $(X_\e, \omega_\e)\in \cM$, the two vertical cylinders bounded by $\{a, a'\}$ and $\{b, b'\}$ respectively cannot be $\cM$-parallel to any other cylinder.  Hence, the union $\cC$  of these two cylinders is either a single equivalence class of $\cM$-parallel cylinders, or the union of two equivalence classes. In either case, the cylinder deformation $a_t^\cC(X,\omega)$ must remain in $\cM$. But this deformation scales $x$ and $y$ while leaving $c$ and $d$ constant. 
\end{proof}

Our next cautionary example is even more basic. Consider the quasi-projective variety $\cM=\{(x,y,t): y^2=x^3+t, t\neq 0\}$, whose closure in affine three-space is given by $\cM\cup \partial \cM$, where  $\partial\cM=\{(x,y,0): y^2=x^3\}$. If one defines the tangent space as the space of derivatives of smooth maps from an interval in $\bR$ into the variety, than the tangent space to $\partial\cM$ at $(0,0,0)$ is zero dimensional; it does not contain the limit of tangent spaces to $\cM$ along a sequence of points converging to $(0,0,0)$. In our situation, this relates to the possible worry that the tangent space to the boundary orbit closure might be smaller than expected. Keep in mind that we do not have any results on whether the closure of an orbit closure in any strata compactification is smooth. But in our situation we at least know that the relevant part of $\partial\cM$ is a properly immersed smooth orbifold, which may account for why this possible worry is comparatively easier to rule out. 

\section{Cylinder deformations}\label{S:parallel}

Recall that two cylinders on $(X,\omega)\in \cM$,\footnote{This notation works well if $\cM$ is an embedded manifold, but in reality $\cM$ may have self-crossings. In general $\cM$ is a manifold which is immersed into the stratum, and it is an abuse of notation to write $(X,\omega)\in \cM$. It would be more precise to write $(X,\omega, V)\in \cM$, where $V\subset H^1(X,\Sigma, \bC)$ is the image of the tangent space to $\cM$ in the stratum. Typically $V=T_{(X,\omega)} \cM$ is determined by $(X,\omega)$, but at points of self-crossing there many be finitely many choices of $V$. The notion of $\cM$-parallel of course depends on this choice. As is typical, we now continue with the abuse of notation.} with core curves $\gamma_1$ and $\gamma_2$, are said to be $\cM$-parallel if there is some $c\in \bR$ such that $\int_{\gamma_1}\omega = c \int_{\gamma_2}\omega$ holds on a neighborhood of $(X,\omega)$ in $\cM$. We will call $c$ the ratio. 

Given a cylinder $C$ on $(X,\omega)$ with an orientation of its core curve, define $u_t^C(X,\omega)$ (resp. $a_t^C(X,\omega)$) to be the result of rotating the surface so $C$ is horizontal and the period of the core curve is positive, applying $u_t$ (resp. $a_t$) just to $C$, and applying the inverse rotation. Here 
$$u_t=\left( \begin{array}{cc} 1 & t \\ 0 & 1\end{array}\right) \quad\text{and}\quad a_t=\left( \begin{array}{cc} 1 & 0 \\ 0 & e^{t}\end{array}\right).$$ 
Define an orientation on a collection of parallel cylinders to be a choice of orientation on each core curve $\gamma_i$; say that the orientation is consistent if the integrals $\int_{\gamma_i} \omega$ are positive multiples of each other.

Given an equivalence class $\cC=\{C_1, \ldots, C_n\}$ of consistently oriented $\cM$-parallel cylinders, we define $u_t^\cC(X,\omega) = u_t^{C_1} \circ \cdots \circ u_t^{C_n}(X,\omega)$, and similarly for $a_t^\cC$. The main result from \cite{Wcyl} gives the following.

\begin{thm}[Cylinder Deformation Theorem]\label{T:CDT}
If $\cC$ is an equivalence class of $\cM$-parallel cylinders on $(X,\omega)\in \cM$, then $$a_s^\cC(u_t^\cC(X,\omega))\in \cM$$
for all $s,t\in \bR$. 
\end{thm}

The rest of this section is devoted to the following theorem, which may be viewed as a black box. It will be used to avoid the situation of the first cautionary example. 

\begin{thm}\label{T:deform}
For each invariant subvariety $\cM$, and each sufficiently large real number $M$, there is some $M'>M$ such that for all  $(X,\omega)\in \cM$, there exists $k\geq 0$, a collection of disjoint cylinders $E_1, \ldots, E_k$ of modulus at least $M$, and $(t_1, \ldots, t_k) \in \bR^k$, such that the path 
$$a_{t\cdot t_1}^{E_1} \circ \cdots \circ a_{t\cdot t_k}^{E_k}(X,\omega), \quad t\in [0,1]$$
is in $\cM$, and at the end of the path all the $E_i$ have modulus at least $M$ and all cylinders on the surface have modulus at most $M'$.
\end{thm}

The cylinders $E_1, \ldots, E_k$ need not be parallel. 
If $(X,\omega)$ has no cylinders of modulus greater than $M'$, we allow $k=0$, so no cylinder deformation is required, and the constant path gives the result. 

Theorem \ref{T:deform}  is closely related to the Cylinder Finiteness Theorem \cite[Theorem 5.1]{MirWri}, and in fact we use the following version of this result. 

\begin{thm}[Cylinder Finiteness Theorem]\label{T:CDT}
For any $\cM$, the set of ratios of $\cM$-parallel cylinders is finite. Furthermore, the set of equations which restrict the heights of $\cM$-parallel cylinders is also finite.
\end{thm}

We use the term ``height" of a cylinder to denote the dimension transverse to the circumference. By ``equations which restrict the heights" we mean more precisely the following, which perhaps would be more fully described by ``equations that locally restrict the heights for deformations which do not produce new $\cM$-parallel cylinders."    For any equivalence class $\cC=\{C_1, \ldots, C_n\}$ of $\cM$-parallel cylinders on any $(X,\omega)\in \cM$ and any $t_1, \ldots, t_n\in \bR$ sufficiently close to 0, we have 
$$a_{t_1}^{C_1}\circ \cdots \circ a_{t_n}^{C_n}(X,\omega)\in \cM$$
if and only if the heights of the $C_i$ on $a_{t_1}^{C_1}\circ \cdots \circ a_{t_n}^{C_n}(X,\omega)$ satisfy certain linear equations depending on $(X,\omega)$, $\cM$ and $\cC$; these are the linear equations referred to above, and, up to reordering the basis, they can be viewed as equations on the vector space $\bR^n$. The second part of the Cylinder Finiteness Theorem says that only finitely many systems of linear equations arise in this way for each $\cM$. 

{\begin{rem}
We can rephrase in the language of \cite[Section 4]{MirWri}. Let $\gamma_i \in H_1(X-\Sigma)$ be the class of the core curve of  $C_i$, and let $\gamma_i^*\in H^1(X,\Sigma)$ be its Poincare dual. If $C_i$ has height $h_i$, then the Cylinder Deformation Theorem exactly gives that $\sum h_i \gamma_i^* \in T_{(X,\omega)}(\cM)$. Conversely if $\sum h_i' \gamma_i  \in T_{(X,\omega)}(\cM)$, then one can deform in a complex multiple of this direction to change the height of each $C_i$ to $h_i + \delta h_i'$, for $\delta\in \bR$ sufficiently small, and this deformation can be expressed in terms of the cylinder deformations $a_{t_i}^{C_i}$ referred to above. Thus, the system of linear equations we refer to above is exactly the system of linear equations determining which linear combinations of the $\gamma_i^*$ are in $T_{(X,\omega)}(\cM)$.
\end{rem}}

This statement of the Cylinder Finiteness Theorem is slightly stronger than \cite[Theorem 5.1]{MirWri}, but it follows from the same proof. We will nonetheless give an outline of a different proof here, using that $\cM$ is an algebraic variety. 

\begin{proof}[Proof of Theorem \ref{T:CDT}]
Let $\cU$ be a small neighborhood in $\cH$ of the locus in $\BHH$ where there are simple poles. We can pick $\cU$ so that $\cM\cap \cU$ contains only finitely many components (compare to the arguments in Section \ref{S:continuity}). We can also pick it so that every $(X,\omega)\in \cU$ has cylinders of enormous modulus corresponding to the simple poles of a nearby point of $\BHH$. 

For each equivalence class $\cC$ on each $(X, \omega) \in \cM$, consider 
$$\lim_{t\to\infty} a_t^\cC(X,\omega) \in \BHH.$$
On this limit, each cylinder of $\cC$ gives rise to a simple pole. (The ratios of the residues are exactly the ratios of the cylinders in $\cC$.) Hence, for $t$ large enough, $a_t^\cC(X,\omega)$ is in one of the finitely many components of $\cM\cap \cU$. 

Each component has only finitely many large modulus cylinders, and there are only finitely many sets of linear equations determining which deformations of the large cylinders stay in $\cM$. There are also only finitely many equations fixing the ratios of these large modulus cylinders. 
\end{proof}

\begin{rem}
In the previous proof, we warn that $\lim_{t\to\infty} a_t^\cC(X,\omega)$ may lie in a higher codimension subvariety of the boundary of $\cM$, so the relationship between  finiteness of ratios  and  finiteness of the number of components of the boundary is less obvious than it may seem at first. A related warning is that, for an arbitrary point of the boundary of $\cM$, there may be infinite cylinders (simple poles) that are $\cM$-parallel to finite cylinders. Neither phenomenon occur unless $\cC$ admits ``non-standard" cylinder deformations (see \cite{MirWri} for the definition). The simplest relevant example is when $\cM$ is a stratum and $\cC$ is a pair of homologous cylinders. 
\end{rem}

Now we give an elementary lemma. 

\begin{lem}\label{L:LinAlg}
For each linear subspace $V\subset \bR^n$, and each $H>1$, there is some $H'>H$ such that for every $v\in V\cap [0,\infty)^n$, there exists $v'\in V$ such that every coordinate of $v$ less than $H$ is equal to the corresponding coordinate of $v'$, and every other coordinate of $v'$ is in $[H, H']$.
\end{lem}

\begin{proof}
Consider the set of all $v\in V\cap [0,\infty)^n$ where a fixed subset of the coordinates are at most $H$. For each such $v$, let $v'$ be a vector that minimizes $\|v'\|_\infty$ subject to the condition that every coordinate of $v$ less than $H$ is equal to the corresponding coordinate of $v'$. The function $v\mapsto \|v'\|_\infty$ is easily seen to be upper semi-continuous. Since this function only depends on a fixed subset of the coordinates of $v$, all of which are in $[0,H]$, it can be viewed as a function on a compact set, and is hence bounded by some number $H'$. 
\end{proof}

\begin{proof}[Proof of Theorem \ref{T:deform}]
Let $R$ be the maximum ratio of two $\cM$-parallel cylinders. Let $H'$ be the maximum of the $H'$ given by Lemma \ref{L:LinAlg} for each of the finitely many systems of linear equations given by Theorem \ref{T:CDT}, with $H=R M$. Finally, set $M'= H' R$. 

Let $\cC=\{C_1, \ldots, C_n\}$ be an equivalence class of $\cM$-parallel cylinders that contains at least one cylinder of modulus at least $M$, say $C_1$. Consider the vector $\vec{h}\in \bR^n$ of the heights of these cylinders.
Then consider the normalized vector $v= \vec{h}/c(C_1)$, whose first entry is the modulus $m(C_1)=h(C_1)/c(C_1)$ of $C_1$. Lemma \ref{L:LinAlg} with $H=R M$, applied to $v$, gives a $v'$ with all coordinates at most $H'$ and all of the changed coordinates at least $RM$. If we define $\vec{h}'=c(C_1) v'$, then all entries of $\vec{h}$ are at most $H' c(C_1)$, and all the changed coordinates (compared to $\vec{h}$) are at least $RM c(C_1)$. This shows that we can deform the cylinders in $\cC$ with height greater than  $H c(C_1)=RM c(C_1)$ so that they have height in $[ RMc(C_1), H' c(C_1)]$. 

Any cylinder in $\cC$ with height greater than $RM c(C_1)$  has, before the deformation, modulus at least $M$. After the deformation, it has modulus in $[M , H' R]$. 

Since cylinders of modulus greater than a universal constant are disjoint, this gives the result. 
\end{proof}

\section{The boundary of an orbit closure}

Before we discuss Theorem \ref{T:main}, we should clarify the basic context for our discussion. Note that, in the context of multi-component surfaces, it has not been proven that every $GL^+(2,\bR)$ orbit closure is a variety. 

\begin{lem}
The boundary of an invariant subvariety $\cM$ in a stratum of $\WHH_{<\infty}$ is an invariant subvariety, and is defined by linear equations in local period coordinates in the boundary stratum. 
\end{lem}

\begin{proof}[Sketch of proof.]
We first see that it is an algebraic variety. To start, consider a point $(X,\omega)$ in the boundary of $\cM$ in $\BHH$, and assume $\omega$ does not have any poles. Here $X$ is a nodal Riemann surface, and $\omega$ may be zero on some components. 

For each component $(X_i, \omega_i)$ of $(X,\omega)$, let $g_i$ be the genus of $(X_i, \omega_i)$, and $n_i$ be the number of marked points and nodes. Define $P_i$ to be $\cM_{g_i, n_i}$ if $\omega_i=0$, and otherwise define $P_i$ to be the stratum of $(X_i, \omega_i)$. There is a natural map $G\colon \prod_i P_i \to \BHH$ whose image contains $(X,\omega)$, defined by gluing the surfaces in the $P_i$ in the same configuration as $(X,\omega)$. 

Let $\overline{\cM}$ denote the closure of $\cM$ in $\BHH$. This is a variety, hence so is its preimage  $G^{-1}(\overline{\cM})$. Let $I$ denote the set of $i$ for which $\omega_i=0$. Then we may write $\prod P_i = \prod_{i\in I} P_i \times \prod_{i\notin I} P_i$ as the product of two varieties, and we may consider the projection 
$$\pi_{nz} \colon \prod P_i \to  \prod_{i\notin I} P_i.$$ 
Then $\pi_{nz}(G^{-1}(\overline{\cM}))$ is a component of the boundary of $\cM$ in a stratum of $\WHH_{<\infty}$, and all components arise in this way. 

Given a subvariety of a product of varieties, its image under projection to a factor is a constructible set. Since the boundary is by definition closed, this shows it is a variety.

The boundary is $GL^+(2,\bR)$ invariant since $\cM$ is, and since the $GL^+(2,\bR)$ action extends continuously to the boundary. Hence a folklore observation of Kontsevich, recorded in \cite[Proposition 1.2]{M6}, gives that the boundary is locally defined by linear equations in period coordinates. 
\end{proof}

We start our proof of Theorem \ref{T:main} by applying Theorem \ref{T:deform}.  

\begin{lem}\label{L:deformlimit}
Suppose that $(X_n, \omega_n)\in \cM$ converge to $(X_\infty, \omega_\infty)\in \WHH_{<\infty}$. Let $M$ be strictly greater than the modulus of any cylinder on $(X_\infty, \omega_\infty)$, and let $M'$ be given by Theorem \ref{T:deform}. Then the $(X_n', \omega_n')$  given by Theorem \ref{T:deform} also converge to $(X_\infty, \omega_\infty)$ in $\WHH$. 
\end{lem}

{Here, for each $n$, $(X_n', \omega_n')$ is the endpoint of the path produced by  Theorem \ref{T:deform} that starts at $(X_n, \omega_n)$.}

\begin{proof}
Let $\e_n$ be the largest circumference of a cylinder of modulus greater than $M$ on $(X_n, \omega_n)$. It is easy to see that $\e_n\to 0$; see \cite[Section 2.4]{MirWri} for more details. 

Since only disjoint cylinders of smaller and smaller circumference are modified in the passage from $(X_n, \omega_n)$ to $(X_n', \omega_n')$, the two sequences have the same limit in $\WHH$. 
\end{proof}

\begin{cor}\label{C:suffices}
It suffices to prove Theorem \ref{T:main} under the assumption that there is some $M'$ such that  the $(X_n, \omega_n)$ have no cylinders of modulus greater than $M'$.
\end{cor}

\begin{proof}
First note that, assuming $M$ is large enough, two cylinders of modulus at least $M$ cannot cross. Thus, passing from $(X_n, \omega_n)$ to $(X_n', \omega_n')$ removes the large modulus cylinders without creating any new ones, since even at the end of the deformation the large modulus cylinders have modulus at least $M$. 

Next, note that the cylinder deformation provides a path in $\cM$ from $(X_n, \omega_n)$ to $(X_n', \omega_n')$. Along this path, $T(\cM)$ and $T(\cM)\cap \Ann(V)$ are constant. 
\end{proof}

We now give a result which, together with the above, will rule out the problem of the first cautionary example in Section \ref{S:caution}. 

\begin{thm}\label{T:continuity} 
For any $(X_\infty, \omega_\infty)$ in the boundary of $\WHH_{<\infty}(\kappa)$, there exist finitely many connected, simply connected, open subsets  $S_i$ of $\bC^{\dim \cH(\kappa)}$ with continuous maps $S_i\to \cH(\kappa)$ such that:
\begin{enumerate}
\item There is a  family of translation surfaces over each $S_i$ whose fiber over each point is the translation surface represented by 
the image of that point in $\cH(\kappa)$. 
\item  There is a locally constant map $H_1(X,\Sigma)\to H_1(X_\infty,\Sigma_\infty)$ for each $S_i$ from the relative homology of the fibers of this family to that of $(X_\infty, \omega_\infty)$, such that for any sequence in $S_i$ converging to $(X_\infty, \omega_\infty)$, the collapse maps discussed in Section \ref{S:compactification} can be chosen to induce this map on $H_1$. 
\item With respect to this map, the relative and absolute periods extend continuously from each $S_i$ to $(X_\infty, \omega_\infty)$. (The relative cohomology of $(X_\infty, \omega_\infty)$ is a quotient of the relative cohomology of a translation surface in $S_i$, so the period of each relative homology class can be defined at $(X_\infty, \omega_\infty)$.)
\item Let $\cU_M$ denote the subset of $\cH(\kappa)$ without cylinders of modulus greater than $M$. Then for $M$ sufficiently large, the image of the union $\cup S_i$ contains the intersection of $\cU_M$ with a neighborhood of $(X_\infty, \omega_\infty)$ in $\WHH(\kappa)$.
\item For each $i$, any subvariety $\cM\subset \cH(\kappa)$ intersects the image in $\cH(\kappa)$ of $S_i$ in at most finitely many connected components. 
\end{enumerate}
\end{thm}

We remark that the above subsets $S_i$ arise from neighborhoods of preimages of $(X_\infty,\Sigma_\infty)$ in a smooth compactification of $\HH(\kappa)$. See Section \ref{S:continuity} for an outline of the proof. The main purpose of all the above is to get the following result. 

\begin{lem}
{If a sequence $(X_n, \omega_n)\in \cM$ converges to $(X_\infty, \omega_\infty)$ in $\WHH(\kappa)$},  then $\omega_\infty \in T(\cM)$, using the identification between $H^1(X_\infty, \Sigma_\infty)$ and $\Ann(V)$. 
\end{lem} 

\begin{proof}
It suffices to assume that we are in the situation of Corollary \ref{C:suffices}. By partitioning into finitely many subsequences, it also suffices to assume that all $(X_n, \omega_n)$ are in one of the $S_i$ of Theorem \ref{T:continuity}, and moreover in one of the components of $\cM\cap S_i$. 

On this component, $T(\cM)$ is constant, and is cut out by finitely many equations on periods. By Theorem \ref{T:continuity}, the periods extend continuously to $(X_\infty, \omega_\infty)$, so this gives the result. 
\end{proof}

\begin{rem}
The previous lemma is treated rather indirectly in \cite{MirWri}. There, it is shown that certain tangent vectors to the orbit closure of $(X_\infty, \omega_\infty)$ are contained in $T(\cM)$. These tangent vectors span a space that contains $\omega$. 
\end{rem}

\begin{lem}
In the situation of the previous lemma, for any sufficiently small $\xi\in T(\cM)\cap\Ann(V)$, we have that $(X_\infty, \omega_\infty)+\xi$ is contained in the boundary of $\cM$. 
\end{lem} 

\begin{proof}
This is handled in the same way as \cite{MirWri}, by using Proposition \ref{P:xi} directly.
\end{proof}

\begin{proof}[Proof of Theorem \ref{T:main}]
This follows directly from the previous two lemmas. Indeed, the boundary is  a variety and is $GL^+(2,\bR)$ invariant. The previous two lemmas show that its tangent space is ``not too big" and ``not too small" respectively. 
\end{proof}

\section{The structure of the multi-component boundary}\label{S:multi}

{In the first part of this section, we discuss to what extent basic dynamical results on orbit closures extend to the multi-component case. In the second part, we discuss and prove Theorem \ref{T:prime}. 

The length of this section is due to the proof of the second statement of Theorem \ref{T:prime}, and the reader may wish to first consider the case when there are only two components, when the first subsection is not required and the technical difficulties in the proof do not appear.}

{\subsection{Preliminary dynamical results}\label{SS:multidyn} In this subsection we discuss invariant measures and hyperbolicity of the diagonal Teichm\"uller geodesic flow. We will use these preliminary results in the proof of Theorem \ref{T:prime}, where we will claim that a basic result on invariant subvarieties of connected surfaces can partially extend to the multi-component case. The results of this subsection justify this, and may also be of independent interest. 

\begin{lem}\label{L:measure}
Every invariant subvariety $\cM$ of multi-component surfaces has a $GL^{+}(2,\bR)$ invariant Lebesgue class measure. 
\end{lem}

Since this measure is invariant under $GL^{+}(2,\bR)$, it necessarily has infinite mass.

\begin{proof}
Let $(X,\omega)\in \cM$. If the monodromy representation $\pi_1(\cM) \to GL^{+}(T_{(X,\omega)}(\cM))$ has image contained in the matrices with determinant $\pm 1$, we may simply use (any multiple of) the standard Lebesgue measure on the tangent space to $\cM$, which provides local coordinates for $\cM$. So, we will show this is the case. 

The identity component of the Zariski closure of the image of the full monodromy representation on absolute cohomology is semisimple by \cite[Corollaire 4.2.9]{DeligneII}. Standard Lie theory gives that it acts on any invariant subspace by matrices with determinant $	1$, since a connected algebraic group is semisimple if and only if it is equal to its derived subgroup \cite[Section 14.2, page 182]{Borel}. In particular, the identity component of the Zariski closure of monodromy acts on $p(T_{(X,\omega)}(\cM))$ by matrices of determinant $1$. Since every element of the monodromy group has a finite power in this identity component, we get that all of monodromy acts on  $p(T_{(X,\omega)}(\cM))$ by matrices of determinant $\pm1$.

Now, recall that $\ker(p)$ is a trivial bundle over a finite cover of $\cM$, and so the same is true for $\ker(p)\cap T(\cM)$; see, for example, \cite[Lemma 2.3]{LNW}. Consider any volume form on $\ker(p)\cap T_{(X,\omega)}(\cM)$, i.e. an element of the top exterior power. Combining this with a volume form on $p(T_{(X,\omega)}(\cM))$ gives rise to a volume form which is locally constant. Monodromy preserves this volume form. 
\end{proof}

\begin{lem}\label{L:finite}
The unit area locus of an invariant subvariety $\cM$ of multi-component surfaces has a locally finite $SL(2,\bR)$ invariant measure, which, for any $\e>0$, is finite on the invariant subset where all components of the surface have area at least $\e$.
\end{lem}

Here unit area means that the sum of the areas of the components is 1. The proof is a slight modification of the proof of \cite[Corollary 2.6]{MinW}, and we include it for the convenience of the reader. 

\begin{proof}
Let $\mu$ be a measure produced by Lemma \ref{L:measure}. We define a measure $\mu_0$ on the unit area locus using the standard technique, namely 
$$\mu_0(S) = \mu( \{ (X, t\omega) : (X,\omega)\in S, t\in (0,1) \}).$$

Let $A_\e$ be the subset of the unit area locus where the area of each component is at least $\e$, and let $\mu_\e$ be the restriction of $\mu_0$ to $A_\e$. We will now show that $\mu_\epsilon$ is a finite measure, proving the lemma. Note that $SO(2)$ invariance, the fact that there are only countably many saddle connection directions on a given surface, and Fubini imply that $\mu_0$-almost every surface does not have any horizontal saddle connections. 

It suffices to assume $\cM$ is a subvariety of a product $\cH_1 \times \cdots\times \cH_k$ of strata (rather than a quotient of this by a subset of the symmetric group), so for notational convenience we will do so. Let 
$$u_t =\left( \begin{array}{cc} 1 & t \\ 0 & 1\end{array}\right).$$
By \cite[Theorem H2]{MinW}, there is a compact subset $K_i$ of the unit area locus in $\cH_i$ so that the $u_t$ orbit of every surface without horizontal saddle connections spends (asymptotically) at least $1-1/(2k)$ of its time in $K_i$. 

Let $K = \prod_{i=1}^k [\epsilon, 1] K_i$ be the compact subset of $\cH_1 \times \cdots \times\cH_k$ where all components have area at least $\e$ and at most 1, and, after normalizing the area, each component lies in the appropriate set $K_i$. We get immediately that, for the diagonal $u_t$ action on $\cH_1 \times \cdots\times \cH_k$, any point in $A_\e$ without horizontal saddle connections has $u_t$ orbit that spends at least half of its time in $K$. 

Let $f \in L^1(\mu_\e)$ be a positive, continuous function, which can be constructed using local finiteness of the measure.  Let $\overline{f}(X,\omega)$ denote the time average of $f$ for the forward  $u_t$ orbit of $(X,\omega)$. The Birkhoff Ergodic Theorem implies that $\overline{f}$ is defined almost everywhere and is in $L^1(\mu_\e)$; this does not require finiteness or ergodicity of the measure. 

Let $\alpha$ be the minimum of $f$ on $K$. The recurrence results imply that $\overline{f} \geq \alpha/2$ on a full measure subset for $\mu_\e$, so the fact that $\overline{f}$ is integrable implies $\mu_\e$ has finite measure. 
\end{proof}

\begin{lem}\label{L:hyperbolic}
Suppose $\cM$ is a prime invariant subvariety of multi-component surfaces, and moreover satisfies the second conclusion of Theorem \ref{T:prime}. Then the diagonal action of  $g_t=\left( \begin{array}{cc} e^{t} & 0 \\ 0 & e^{-t}\end{array}\right)$
is uniformly hyperbolic on compact subsets of the unit area locus in $\cM$. 
\end{lem}

The statement will be clarified in the proof. This proof and the next assume some familiarity with the theory of the Teichm\"uller geodesic flow in the connected case; a suitable introduction is the survey \cite[Sections 3,4]{ForniMatheus:Intro}. For example, the reader should be familiar with the fact that each Lyapunov exponent $L$ of the Kontsevich-Zorich cocycle contributes two Lyapunov exponents  to the Teichm\"uller geodesic flow, namely $L \pm 1$. }

{\begin{proof}
Again we assume $\cM\subset \cH_1 \times \cdots \times\cH_k$. 

We will build on the connected case. Suppose $(Y, \eta)$ is a connected unit area translation surface. Then, if $\| \cdot \|_t$ denotes the relative Hodge norm on the relative cohomology of $g_t(Y,\eta)$, then it is known that $\|[\Im(\eta)]\|_t = e^t$, and, if $\lambda$ is a relative cohomology class such that $p(\lambda)$ is symplectically orthogonal to $\Re(\eta)$, and the geodesic $\{g_t(Y,\eta), t\in [0,T]\}$ starts, ends, and spends at least half its time in a compact set, then $\|\lambda\|_t \leq e^{(1-\alpha)t}$, where $\alpha>0$ depends on the compact set. See \cite[Theorem 7.3]{EMR} for this statement, which implies the uniform hyperbolicity on compact sets, and the references in \cite{EMR} to work such as \cite{V, F, AF} with previous similar statements.  

We wish to show a similar statement for $(X,\omega) = ((X_1, \omega_1), \ldots, (X_k, \omega_k))$, using the norm on relative cohomology which is  the sum of the relative Hodge norms in each component. This will follow formally from the connected case once we establish the following claim: If $(\lambda_1, \ldots, \lambda_k) \in p(T_{(X,\omega)}(\cM))$, then for any $1\leq i, j \leq k$ we have that $\lambda_i$ is symplectically orthogonal to $\Re(\omega_i)$ if and only if $\lambda_j$ is symplectically orthogonal to $\Re(\omega_j)$. 

By continuity, it suffices to show this for $(X,\omega)$ in a dense set. So, using Lemma \ref{L:finite} and Poincare recurrence, we can assume that the $g_t$ orbit of $(X,\omega)$ recurs to a compact set.  The maps $(\lambda_1, \ldots, \lambda_k) \mapsto \lambda_i$ are injective by our assumption that $\cM$ satisfies the second conclusion of Theorem \ref{T:prime}, and on this compact set we can assume that they distort norms by at most a constant multiplicative factor. If $\lambda_i$ is symplectically orthogonal to $\Re(\omega_i)$, its norm grows at most like $e^{(1-\alpha)t}$. The bounded norm distortion then implies that, on the unbounded returns to this compact set, the norm of $\lambda_j$ is $O(e^{(1-\alpha)t})$. But if $\lambda_j$ were not symplectically orthogonal to $\Re(\omega_j)$, it would be a non-zero multiple of $\Im(\omega_j)$ plus something orthogonal, and the norm would grow like a constant times $e^t$, giving a contradiction. 

Now, to conclude the proof, define a subspace of $H^1(X, \Sigma,i\bR)\cap T_{(X,\omega)}(\cM)$ consisting of those $v=(v_1, \ldots, v_k)$ for which $p(v_1)$ is symplectically orthogonal to $\Re(\omega_1)$. Any  such $v$ is a tangent vector to the stratum, and will have uniform exponential decay  when pushed forward by the derivative of Teichm\"uller geodesic flow; it hence is in the stable manifold of $(X,\omega)$. The exponential decay comes from the $e^{-t}$ in the definition of Teichm\"uller geodesic flow, which is definitively more powerful than the growth $O(e^{(1-\alpha)t})$ of $v$ under parallel transport. Note that the subspace has dimension $\frac12 \dim_\bR T_{(X,\omega)}(\cM)-1$; double this (for stable and unstable) plus 1 (for the flow direction) is the dimension of the unit area locus in $\cM$. 
\end{proof}}

{\begin{cor}\label{C:ergodic}
Under the same assumptions as in Lemma \ref{L:hyperbolic}, the $g_t$ action on the unit area locus is ergodic. In particular, the ratios of areas of components is locally constant on the unit area locus. 
\end{cor}

\begin{proof}
We first show ergodicity on the locus where each component has area at least $\e$, which is finite measure. There the result follows from the Hopf argument and Lemma \ref{L:hyperbolic}; see for example \cite[Sections 3,4]{ForniMatheus:Intro} for an expository treatment. Now, ergodicity on this locus implies that the ratio of areas must be constant, since the ratio is preserved by the flow. Since this is true for all $\e>0$, the ratios are in fact simply constant.     
\end{proof}

\begin{cor}\label{C:EV}
Under the same assumptions as in Lemma \ref{L:hyperbolic}, closed $g_t$ orbits are dense in $\cM$. The associated induced monodromy actions on relative cohomology have that the largest and smallest eigenvalues, namely $e^t$ and $e^{-t}$, have multiplicity 1.
\end{cor} 

\begin{proof}
The first statement follows from  Lemma \ref{L:hyperbolic} and ergodicity, using a closing lemma, as in the connected case discussed in \cite[Theorem 8]{EMR}. The multiplicity 1 statement follows from Lemma \ref{L:hyperbolic}.  
\end{proof}

\begin{rem}
The above results will be used in the proof of the second statement of Theorem \ref{T:prime}, which is essentially by induction on the number of components. Once that is concluded, we will know the above results for all prime invariant subvarieties. Note that a product invariant subvariety can never have that the diagonal $g_t$ action is hyperbolic, because it is part of an $\bR^2$ action. 

Furthermore, by Corollary \ref{C:ergodic} and Lemma \ref{L:finite}, the $SL(2,\bR)$ invariant measure on any prime invariant subvariety is finite. Since every invariant subvariety is (essentially) a product of primes, and the measure is a product, we get that the $SL(2,\bR)$ invariant measure is always finite, without need to restrict to the locus where all components of the surface have area at least $\e$.
\end{rem}}

\subsection{Theorem \ref{T:prime}}

We begin with the easiest part of Theorem \ref{T:prime}.

\begin{proof}[Proof of Theorem \ref{T:prime}, first statement]
Suppose that $$\cN\subset \cH_1\times \cdots\times \cH_k,$$ where each $\cH_i$ is a stratum of single component surfaces. 

{Set $\cN_i= \overline{\pi_i(\cN)}$. Note that $\pi_i(\cN)$ is a constructible set, so its Zariski closure is equal to its analytic closure. Set $\cN_i' = \overline{\cN_i\setminus \pi_i(\cN)}$. Both $\cN_i$ and $\cN_i'$ are invariant subvarieties, so this gives the result. }
\end{proof}

We now give an example to show that $ \pi_i(\cN)$ need not be closed. 

\begin{ex}
Let $\cM$ denote the set of pairs 
$$((X, \omega, \{p_1\}), (X,\omega, \{p_2, p_3\}))\in \cH(2,0)\times \cH(2,0,0)$$
such that $p_1$ is equal to $p_2$ or $p_3$, and $p_2$ and $p_3$ are interchanged by the hyperelliptic involution. Then the projection of $\cM$ to the first factor is equal to $\cH(2,0)$ minus the locus where the marked point $p_1$ is a Weierstrass point. 
\end{ex}

{\begin{rem} $\cN_i$ is locally described by a finite union of linear subspaces. At a generic point, there is just a single linear subspace; the other points are called self-crossing points. If $\cN_i$ has no self-crossing points, then one can show using period coordinates that $\pi_i(\cN)$ is open in $\cN_i$ and hence is equal to $\cN_i \setminus \cN_i'$. Even if $\cN_i$ has self-crossings, it is possible to recover such a statement at the cost of lifting to the smooth orbifold whose properly immersed image is $\cN_i$. 
\end{rem}}

\begin{lem}
An invariant subvariety $\cM$ is not prime if and only if  some (equivalently, any) $(X,\omega)\in \cM$ can be written as $(X,\omega)=(X_1, \omega_1)\cup (X_2, \omega_2)$, where the $(X_i, \omega_i)$ are disjoint and are a non-empty union of components, and we have 
$$T_{(X,\omega)}{(\cM)} = (T_{(X,\omega)} (\cM) \cap H^1(X_1, \Sigma_1)) \oplus (T_{(X,\omega)}{(\cM)} \cap H^1(X_2, \Sigma_2)),$$
where $\Sigma_i$ is the set of zeros of $\omega_i$. 

In this case, if $\cM'$ is the minimal cover of $\cM$ on which we can divide the components into two subsets consistent with $(X,\omega)=(X_1, \omega_1)\cup (X_2, \omega_2)$, then $\cM' = \cM_1\times \cM_2$ for invariant subvarieties $\cM_1, \cM_2$.      
\end{lem}

More concisely, $\cM$ is not prime if and only if its tangent space is a direct sum in a way compatible with the decomposition into components, and it suffices to check this at a single point. 

\begin{proof}
Suppose we have the direct sum decomposition of $T_{(X,\omega)}{(\cM)}$.
By definition of $\cM'$, we have $\cM'\subset \cH_1\times \cH_2$, where $\cH_i$ are strata of multi-component surfaces, and $(X_i,\omega_i)\in \cH_i$. 

Let $\pi_i\colon \cH_1\times \cH_2\to \cH_i$ be the coordinate projections, and define $\cM_i=\overline{\pi_i(\cM')}$, so $\cM'\subset \cM_1\times \cM_2$. Using period coordinates we can see that $\cM'$ contains an open subset of  $\cM_1\times \cM_2$. Because both varieties are irreducible and closed, $\cM'= \cM_1\times \cM_2$.
\end{proof}

\begin{cor}
Any invariant subvariety $\cM$ of a product of strata is a product of prime invariant varieties and the product decomposition is unique. 
\end{cor}

\begin{proof}
The previous lemma gives that $\cM$ can be written as a product if it is not prime, and we can iterate this to get that $\cM$ is a product of prime invariant varieties. The uniqueness is an exercise in linear algebra.
\end{proof}

\begin{ex}
Let $\cH$ be a stratum. Then the diagonal embedding of $\cH$ in $\cH\times \cH$ is prime. 
\end{ex}

\begin{ex}
Let $\cH$ be a stratum. Consider the locus $\cM$ of pairs $((X_1, \omega_1), (X_2, \omega_2))$ for which there exists some $(X,\omega)\in \cH$ such that $(X_1,\omega_1)$ is a triple cover of $(X,\omega)$ simply branched over four points and $(X_2, \omega_2)$ is a double cover of $(X,\omega)$ simply branched over two points. The branch points are allowed to be arbitrary points of $X$ that are not zeros of $\omega$. Then $\cM$ is a prime invariant variety, but each of the two coordinate projections from $\cM$ are infinite to one. (None of the choices in this example are important.) 
\end{ex}

To proceed, we  require the following result from \cite{Wfield}. As motivation, let us mention that the proof of the second statement of Theorem \ref{T:prime} will roughly speaking be by induction, and ``smaller" prime orbit closures will appear that are known to satisfy this second statement.

\begin{thm}\label{T:simplicity} 
Let $\cM$ be an (irreducible) invariant subvariety of single component translation surfaces, and let $\cM'$ be a proper invariant subvariety (possibly with several components). Then any proper flat subbundle of $T(\cM)$ defined over a finite cover of $\cM\setminus \cM'$ is contained in $\ker(p)$. 

The same statement holds if $\cM$ is a prime invariant subvariety of multi-component translation surfaces satisfying the second conclusion of Theorem \ref{T:prime}. 
\end{thm}

Here ``proper'' just means that the subbundle is not all of $T(\cM)$, and $p$ is the usual map from  relative cohomology to absolute cohomology, both viewed as bundles over $\cM$.  In \cite{Wfield} the theorem is stated for bundles over $\cM$ rather than $\cM\setminus \cM'$, but the proof gives the more general statement. Similarly it is not stated for the prime (multi-component) case, but the proof is identical given the results of Section \ref{SS:multidyn}.

{
\begin{cor}\label{C:fix}
Suppose that $\cM$ is an invariant subvariety of  $\cM_1\times \cM_2$ whose projection to each factor is dominant, where $\cM_1$ is an invariant subvariety of single component translation surfaces, and $\cM_2$ is an invariant subvariety of multi-component surfaces. Let $$(X, \omega) = ((X_1, \omega_1), (X_2, \omega_2))\in \cM.$$ Suppose that the subspace $$p(T_{(X,\omega)}(\cM)) \subset p(T_{(X_1,\omega_1)}(\cM_1)) \oplus  p(T_{(X_2,\omega_2)}(\cM_2))$$ contains a  vector which is zero in the second component but non-zero in the first. Then $\cM=\cM_1\times \cM_2$.

The same statement holds if $\cM_1$ is a prime invariant subvariety of multi-component translation surfaces satisfying the second conclusion of Theorem \ref{T:prime}. 
\end{cor}}

{The proof will using the following topological results from algebraic geometry, which can be found in \cite[Proposition 2.10]{Kollar}.
\begin{prop}\label{P:agtop}
The following statements hold:
\begin{enumerate}
\item Let $Y$ be a connected normal variety, and let $Z\subset Y$ be Zariski closed. If $\iota\colon Y\setminus Z\to Y$ denotes the inclusion, then $i_*(\pi_1(Y\setminus Z))=\pi_1(Y)$. 
\item Let $f\colon X\to Y$ be a surjective morphism between irreducible algebraic varieties, with $Y$ normal. Then $f_*(\pi_1(X))$ is finite index in $\pi_1(Y)$. 
\end{enumerate}
\end{prop}}

{\begin{proof}[Proof of Corollary \ref{C:fix}]
The statement gives the existence of a vector $(v_1, v_2) \in T_{(X,\omega)}(\cM)$ with $v_1\notin \ker(p)$ but $v_2\in \ker(p)$.

Recall that $p(T(\cM_1))$ has totally irreducible monodromy. (A group of matrices is called totally irreducible if every finite index subgroup is irreducible.)  This is implied by Theorem \ref{T:simplicity}, and implicitly applies for to monodromy over the smooth variety whose immersed image is $\cM_1$. (The proof does not use monodromy of loops that jump between branches of $\cM_1$ along the self-crossing locus.) Recall that smooth varieties are in particular normal. 

We claim that Proposition \ref{P:agtop} thus implies that there is a loop in $\cM$ with monodromy not fixing $p(v_1)$. Indeed, using Proposition \ref{P:agtop}, we may delete Zariski closed subsets $Z, Z_1$ of $\cM$ and $\cM_1$ so that $\cM\setminus Z, \cM_1\setminus Z_1$ are smooth, and there is a surjective map $\cM\setminus Z\to \cM_1\setminus Z_1$. Proposition \ref{P:agtop} (2) gives that the image of $\pi_1(\cM\setminus Z)$ is finite index in $\pi_1(\cM_1\setminus Z_1)$, and Proposition \ref{P:agtop} (2)  and the totally irreducibility gives that this finite index subgroup does not fix $p(v_1)$. 

Recall that $\ker(p)$ is a trivial bundle on a finite cover of the stratum, see for example \cite[Lemma 2.3]{LNW}. Modifying the previous argument to use finite covers (or more concretely lifts to strata with zeros labelled), we may obtain a loop in $\cM$ whose monodromy doesn't fix $p(v_1)$ but does fix $v_2$. Subtracting off the image of this monodromy applied to $(v_1, v_2)$, we might as well assume $v_2=0$. 

Let $K$ denote the flat vector subbundle over $T(\cM)$ whose fiber at $(X,\omega)=((X_1, \omega_1), (X_2, \omega_2))$ is the kernel of the projection   $T_{(X,\omega)}(\cM)\to T_{(X_2,\omega_2)}(\cM_2)$. This is a priori a vector bundle on $\cM$, but because it is locally constant and each fiber of the projection has only finitely many components, it gives a vector bundle $K'$ defined over a finite cover of $\cM_1$. The existence of $(v_1, 0)$ shows that $K'$ is non-zero, so the previous theorem gives that $K'$ is all of $T(\cM_1)$ (lifted to the finite cover), which gives the result.
\end{proof}}

\begin{proof}[Proof of Theorem \ref{T:prime}, second statement]
Suppose that $$\cN\subset \cH_1\times \cdots\times \cH_k$$ is a counterexample. Assume $k$ is as small as possible. 

Let $\pi_i\colon \cN\to \cN_i$ be the projection onto the $i$-th factor, and let $\pi_{\hat{i}}$ be the projection away from the $i$-th factor, so $\pi_{\hat{i}}(\cN) \subset \prod_{j \neq i} \cH_j$. Let 
$$(X,\omega)=((X_1, \omega_1), \ldots,  (X_k, \omega_k)) \in \cM,$$
and let $E_i$ denote the kernel of the projection   $$p(T_{(X,\omega)}(\cM))\to p(T_{(X_i,\omega_i)}(\cM_i)).$$

Our goal is precisely to prove that each $E_i$ is zero. If $k=2$, this follows directly from Corollary \ref{C:fix}, since otherwise we get $\cM=\cM_1\times \cM_2$, contradicting that $\cM$ is prime. 

So assume $k>2$. (The following argument also works for $k=2$, but is rather degenerate and overcomplicated in this case.) Without loss of generality, assume that $E_1$ is not zero. By minimality of $k$, we can assume that $$\pi_{\hat{1}} \cN = \cL_1\times \cdots \times \cL_s$$ is a product of prime invariant subvarieties satisfying the second statement of Theorem \ref{T:prime}. Here $s\leq k-1$, and it might be that $s=1$. 

We now claim that for some $i$, $E_1$ contains a vector that is non-zero in the $i$-th component of 
$$p(T_{(X,\omega)}(\pi_{\hat{1}} \cN)) = p(T_{(Y_1,\eta_1)}(\cL_1))\oplus \cdots \oplus p(T_{(Y_s,\eta_s)}(\cL_s))$$
but zero in all the other components, where here each $(Y_i, \eta_i)\in \cL_i$. We will show that this is true for any $i$ for which there exists $v\in E_1$ whose projection to the $p(T\cL_i)$ summand is non-zero.  Indeed, given such a $v$, there is a loop in $\cL_i$ whose monodromy does not fix $v$.  We can lift this loop to a loop in $\cL_1\times \cdots \times \cL_s$ by keeping all of the other coordinates constant. Considering $v$ minus its image under monodromy proves the claim. 

Now, Corollary \ref{C:fix} gives that $\cM$ is not prime, which is a contradiction. (The $\cM_1$ in that corollary corresponds to $\cL_i$ here.) 
\end{proof}

We now turn to the third statement. 

\begin{proof}[Proof of Theorem \ref{T:prime}, third statement]
There is an isomorphism from $p(T\cN)$ to $p(T\cN_i)$ for each $i$. Since by definition rank is half the dimension of fibers of these bundles, this gives the result. 
\end{proof}

\begin{ex}
Consider the locus $\cN\subset \cH(2)$ of pairs $$((X, \omega), (X, \sqrt{2}\omega))$$
for all $(X,\omega)\in\cH(2)$. This $\cN$ is locally cut out by linear equations in $\cQ[\sqrt{2}]$, but both of its projections are a full stratum. This shows that the field $\bk(\cN)$ defined in \cite{Wfield} does not have to be equal to the fields $\bk(\cN_i)$ of the projections $\cN_i$. It is however of course true that $\bk(\cN_i)\subset \bk(\cN)$. 
\end{ex}

For the fourth statement, recall the decomposition 
$$H^1 = \bigoplus p(T\cM)_\tau \oplus \bW,$$
where the sum is over the Galois conjugates of $p(T\cM)$, and $\bW$ is the remaining part of the bundle $H^1$ of absolute cohomology. Filip showed that this direct sum decomposition is compatible with the Hodge decomposition $H^1(X,\bC)=H^{1,0}(X)\oplus H^{0,1}(X)$ \cite{Fi1}. For example, if we set $p(T_{(X,\omega)}\cM)_\tau^{1,0}=p(T_{(X,\omega)}\cM)_\tau \cap H^{1,0}(X)$ etc, we have 
$$p(T_{(X,\omega)}\cM)_\tau = p(T_{(X,\omega)}\cM)_\tau^{1,0} \oplus p(T_{(X,\omega)}\cM)_\tau^{0,1}.$$

Filip proved his result for single component surfaces, and the difficulty in that context is that he does not assume the $GL^{+}(2,\bR)$ orbit closure is a subvariety. Prior to Filip, a similar result was known for any variation of Hodge structure over a quasi-projective base. We now explain  how this result applies to invariant sub-varieties of multi-component surfaces to give the same decomposition. 

\begin{lem}
When $\cM$ is a prime invariant subvariety of multi-component surfaces, there is a decomposition $H^1 = \bigoplus p(T\cM)_\tau \oplus \bW$ compatible with the Hodge decomposition. 
\end{lem}

\begin{proof}
The decomposition follows from the semisimplicity result for monodromy discussed in the proof of Lemma \ref{L:measure}, noting that Theorem \ref{T:simplicity} implies that $p(T\cM)$, and hence also its Galois conjugates, are irreducible. 

We now show that compatibility follows from \cite{D1}; this will demand of the reader some familiarity with variations of Hodge structure. Indeed, \cite[Proposition 1.13]{D1} gives a decomposition 
$$H^1 = \bigoplus_i S_i \otimes W_i,$$
where the $S_i$ are irreducible non-isomorphic variations of Hodge structure and the $W_i$ are vector spaces with Hodge structures. 

Since $p(T\cM)$ is irreducible, it is of the form $S_{i_0} \otimes L_{i_0}$, where $L_{i_0}$ is a line in $W_{i_0}$. (Because of the multiple components, unlike the single component case, $\dim W_{i_0}>1$.) 

Suppose the Hodge decomposition of $S_{i_0}$ included Hodge types $(p, q)$ and $(p', q')$ with $p'> p$, and $W_{i_0}$ included types $(p'', q'')$ and $(p''', q''')$ with $p'' > p'''$. Then 
$$(p,q)+(p'', q''),\quad (p,q)+(p''', q'''),\quad (p', q') + (p''', q''')$$
are distinct. Since these three types must appear in the Hodge decomposition of  $S_{i_0} \otimes L_{i_0}$, which has only types $(1,0)$ and $(0,1)$, we get a contradiction. Hence, one of $S_{i_0}$ or $W_{i_0}$ has pure type, meaning that only a single term appears in its Hodge decomposition. Without loss of generality, up to shifting weights, we can assume that term is $(0,0)$. 

If $S_{i_0}$ had pure type $(0,0)$, then it would have compact monodromy, because its polarization form would be definite. But $p(T\cM)$ cannot have compact monodromy by Corollary \ref{C:EV}. Hence, $W_{i_0}$ has pure type $(0,0)$. Hence $L_{i_0}$ has pure type $(0,0)$, and hence $p(T\cM)=S_{i_0} \otimes L_{i_0}$ is compatible with the Hodge decomposition.

The same argument applies to the Galois conjugates, where non-compactness of monodromy can be established by the presence of uniponents guaranteed by Theorem \ref{T:CDT} (as in \cite[Remark 4.6]{EFW}) and Proposition \ref{P:agtop}. See also \cite{Fi3}. 

A similar argument applies to $\bW$ because it must contain all of $S_i\otimes W_i$ for any $i$ where $W_i$ does not have pure type $(0,0)$. 
\end{proof}

The natural factor of $\Jac(X)$ referenced in Theorem \ref{T:prime} can be defined as 
$$\Jac_\cM(X, \omega) = \left(\bigoplus p(T_{(X,\omega)}\cM)_\tau^{1,0}\right)^* \big/ \left(\bigoplus p(T_{(X,\omega)}\cM)_\tau\right)^*_\bZ,$$ 
where we define $\left(\bigoplus p(T_{(X,\omega)}\cM)_\tau\right)^*_\bZ$ to be the set of linear functionals on 
$\bigoplus p(T_{(X,\omega)}\cM)_\tau$  that take integer values on $H^1(X, \bZ)\cap \bigoplus p(T_{(X,\omega)}\cM)_\tau$. 

\begin{proof}[Proof of Theorem \ref{T:prime}, fourth statement]
By the second statement, we get an isomorphism from $p(T\cN)^{1,0}$ to $p(T\cN_i)^{1,0}$. Combining this with the Galois conjugate versions, we get the result. 
\end{proof}

\section{Proof of Theorem \ref{T:continuity}}\label{S:continuity}

Let $\wideparen{\HH}(\kappa)$ be the  compactification of $\cH(\kappa)$  constructed in~\cite{lms}, which is a smooth complex orbifold.\footnote{The notation $\wideparen{\HH}(\kappa)$ is different from~\cite{lms}, whose notation $\Xi\BM_{g,n}(\kappa)$  is a bit too heavy to carry over.}  Roughly speaking, $\wideparen{\HH}(\kappa)$ parameterizes twisted differentials compatible with a full level graph and with matchings of horizontal directions at the nodes, where the top level differentials are not projectivized (see also~\cite{Many} for a detailed definition of twisted differentials).  

 Note that $\wideparen{\HH}(\kappa)$ admits a continuous and surjective map to $\WHH(\kappa)$ by forgetting lower level differentials (i.e. components of area zero). To prove Theorem \ref{T:continuity}, we will first prove an analogous result for $\wideparen{\HH}(\kappa)$.  This will require of the reader some familiarity with $\wideparen{\HH}(\kappa)$, but, having established this analogous result (Theorem \ref{T:lms}), the rest of the proof of Theorem \ref{T:continuity} will be  more self-contained.

Take a boundary point $(X, \omega)$ in $\wideparen{\HH}(\kappa)$ such that $X$ has $N$ levels $0, -1, \ldots, -N+1$ and such that the subsurfaces $X^{(i)}$ in each level $i$ have no simple polar nodes (i.e. no horizontal edges in the level graph).  Consider a small ball neighborhood $U\subset \wideparen{\HH}(\kappa)$ of $X$.
Let $t_i$ be the smoothing parameter for level $i$ of $(X, \omega)$ for 
$i = -1, \ldots, -N+1$. Let $V\subset U$ be the complement of the union of hyperplanes defined by each equation $t_i = 0$.  
Said differently, $V$ parametrizes smooth translation surfaces near $(X, \omega)$.  

Consider subsets of $V$ defined by, for each $i$, prescribing the sign of the real or imaginary part of $t_i$. Thus, for each $i$, we prescribe one of the following four conditions: 
$$\Re(t_i)>0, \ \Re(t_i)<0, \ \Im(t_i)>0, \ \Im(t_i)<0.$$
There are $4^{N-1}$ such subsets of $V$, and we call them $S_j$ for $j =1, \ldots, 4^{N-1}$. 

The set  $V$ is a union of these  subsets $S_j$, and $V$ can be chosen so that the $S_j$ are connected and simply connected. Consequently, over each $S_j$ the relative homology groups of all surfaces in the universal family can be identified as one model.\footnote{The space $\wideparen{\HH}(\kappa)$ is constructed as a complex orbifold (or algebraic stack), and locally one can pass to a finite cover to have a universal family over it; see~\cite[Section 13]{lms}. Technically speaking, we should let $U$ be a ball in this finite cover to avoid orbifold issues, but we omit this distinction here.} The relative homology of $(X,\omega)$ is a quotient of the relative homology group on $S_j$. 

Take any one of the subsets $S_j$ and simply denote it by $S$. 

\begin{thm}\label{T:lms}
The period of any element of the model relative homology group gives a continuous function on $S \cup \{(X,\omega)\}$.

The same is true with $(X,\omega)$ replaced with anything with the same image in $\WHH$.
\end{thm}

\begin{proof}
Let $(X_s, \omega_s)$ be a path in $S$ that converges to $(X,\omega)$ as $s\to 0$.

In the model relative homology group over $S$, choose homology classes $\gamma^{(i)}_{k}$ in each level $i$ subsurface $X^{(i)}$ for $k = 1, \ldots, j_i$ such that their collection forms a basis of the relative homology group $H_1(X^{(i)}\setminus P^{(i)}, Z^{(i)}, \bbZ)$, where $P^{(i)}$ is the set of polar nodes in $X^{(i)}$ and $Z^{(i)}$ is the set of non-polar nodes
together with marked zeros in $X^{(i)}$. Here a node is called polar if it is a pole of the (twisted) differential on $X^{(i)}$, and otherwise it is called non-polar. 

We may define a version of $\gamma^{(i)}_{k}$ even after plumbing (i.e. for $s \neq 0$) by, for those paths $\gamma^{(i)}_{k}$ that end at a non-polar node, defining a ``base point" near each such node whose position depends $s$ and goes continuously to the node as $s\to 0$. We then define  $\gamma^{(i)}_{k}$ when $s\neq 0$ by, roughly speaking, moving its endpoints at non-polar nodes to the nearby base points, so  $\gamma^{(i)}_{k}$ varies slightly with $s$ near the node and ends at the node for $s = 0$. 
 This corresponds to the so-called perturbed period coordinates constructed in~\cite[Section 11.2]{lms} and illustrated in \cite[Example 11.8 and Figure 6]{lms}. 
 
  Reordering if necessary, we can also assume that $\int_{\gamma^{(i)}_1}\omega_s \neq 0$ for all $i$ and all $s\neq 0$.
  
Define
$$\lambda^{(i)}_{k}(s) = \frac{\int_{\gamma^{(i)}_{k}} \omega_s}{\int_{\gamma^{(i)}_1}\omega_s}, $$
which measures the relative ratios of level-$i$ periods in the degeneration to $X$ and which is well-defined also at $s=0$.  
Since there are no simple polar nodes, a local coordinate system of $\wideparen{\HH}(\kappa)$ at $(X, \omega)$ restricted to this family is given by  
$$\big(\{t_{i}(s)\}_{i=0}^{-N+1}; \{\lambda^{(i)}_{k}(s)\}_{k=2}^{j_i}\big).$$
Denote by $t_{\lceil i \rceil} (s)= t_{0}(s)t_{-1}(s)\cdots t_{i}(s)$.\footnote{Technically speaking, if there are polar nodes of pole order bigger than two, then the $t_i$ coordinates encode some extra finite data of matching horizontal directions when plumbing at such nodes.  In that case one should use a suitable power $t_j^{a_j}$ instead of $t_j$ in the product $t_{\lceil i \rceil}$, where the exponents $a_j$ are determined by the level graph.  Since this is only a matter of notation and does not affect the rest argument, we omit the distinction.}   The construction of perturbed period coordinates gives that we may arrange for $t_{\lceil i \rceil}(s)$ to be equal to $\int_{\gamma^{(i)}_{1}}\omega_s$ (see \cite[Equation (11.8)]{lms} and the paragraph preceding it, noting that therein one period in each level is intentionally omitted and can be used as $t_{\lceil i \rceil}$).

For an element $\gamma$ in the model relative homology group, write it as the following linear combination 
$$ \gamma = \sum_{i,k} c_k^{(i)} \gamma^{(i)}_{k} + \sum_{h,l} c^{(h, l)} \rho^{(h, l)}$$
where the $c_k^{(i)}$ and $c^{(h, l)}$ are constant coefficients and where the $\rho^{(h,l)}$ are parts of (a path representative of) $\gamma$ that connect level $h$ (from a base point in that level) to an adjacent level $l$ for $h > l$.  In other words, we decompose $\gamma$ into paths that are contained in a single level together with paths that cross two different levels through the corresponding (plumbed) node.  Then the period of $(X_s, \omega_s)$ along $\gamma$ is 
$$ \lambda (s) = \sum_{k,i} c_k^{(i)} t_{\lceil i \rceil}(s) \lambda^{(i)}_{k}(s) + \sum_{h,l} c^{(h,l)} r^{(h,l)}(s) $$
 where each $r^{(h,l)}(s)$ is a function that depends on $t_i(s)$ for $i$ between $h$ and $l$, and on the choice of base points in the definition of perturbed period coordinates (see~\cite[Example 11.8]{lms} for an illustration). 
  Despite these choices, the functions $r^{(h,l)}(s)$ vary continuously and converge to zero as $s\to 0$, which follows from the construction of $\wideparen{\HH}(\kappa)$. In summary, all the quantities involved in the above expression of the period $\lambda(s)$ vary continuously and converge as $s\to 0$. We thus conclude that the same holds for  $\lambda (s)$. 
 
 The final claim follows from the same analysis.
 \end{proof}

Now we can finish the proof of Theorem~\ref{T:continuity}.  

\begin{proof}[Proof of Theorem~\ref{T:continuity}]
Denote by $W$ the  preimage of $(X_\infty, \omega_\infty)$ in $\wideparen{\HH}(\kappa)$. As discussed in Section~\ref{S:compactification}, $W$ is compact. So is the subset $W_M \subset W$  where there are no simple poles and no cylinders of modulus bigger than a large constant $M$. Hence we can find finitely many $(X^{(k)}, \omega^{(k)})\in W_M$ whose ball neighborhoods $U_k$ associated with Theorem \ref{T:lms} cover a neighbourhood of $W_M$. We will use all of the simply connected subsets (previously denoted $S_j$ and then by $S$) associated to all the $(X^{(k)}, \omega^{(k)})$. 

By their construction, each such set $S$ is simply connected and the universal family over $S$ has a model homology group for its fibers, thus satisfying (1) and (2).

We now claim that (3) follows from Theorem~\ref{T:lms}. Otherwise, for some $\e>0$ and some element of the model relative homology group, there would be a sequence $(X_n, \omega_n)$ in one of the sets $S$, such that this sequence converges to $(X_\infty, \omega_\infty)$ and such that the value of the corresponding period at each $(X_n, \omega_n)$ is at least $\e$ different from the value of the corresponding period at $(X,\omega)$. Passing to a subsequence if necessary, we can assume that that $(X_n, \omega_n)$ converges in $ \wideparen{\HH}(\kappa)$,  contradicting Theorem~\ref{T:lms}. 

To see (4), suppose otherwise that there are surfaces $(X_n, \omega_n)$ without cylinders of modulus greater than $M$ converging to $(X,\omega)$ but not contained in any of the sets $S$ above. Passing to a subsequence, we can assume that they converge in $\wideparen{\HH}(\kappa)$. Since the limit must be in $W_M$, this contradicts the fact that the union of the $U_k$  cover a neighbourhood of $W_M$.

Finally, (5) follows from the fact that the intersection of $S_j$ with the smooth locus of $\cM$ is a relatively compact semi-analytic set (see e.g.~\cite[Corollary 2.7, Lemma 3.4]{BM88}).
\end{proof}

\bibliography{mybib}{}
\bibliographystyle{amsalpha}
\end{document}